\documentclass{article}
\usepackage{amsmath,amssymb,amscd, float, verbatim, latexsym,cite,wrapfig}
\usepackage{amsthm}
\usepackage{graphicx}
\usepackage[all]{xy}
\usepackage{lscape,subcaption}
\usepackage{color,enumerate,setspace,url}		

\usepackage[margin=1.5in]{geometry}

\usepackage{wrapfig}

\usepackage{color}

\newtheorem{theorem}{Theorem}[section]
\newtheorem{lemma}[theorem]{Lemma}
\newtheorem{conjecture}[theorem]{Conjecture}
\newtheorem{proposition}[theorem]{Proposition}

\newtheorem{definition}[theorem]{Definition}
\newtheorem{algorithm}[theorem]{Algorithm}

\newcommand{\K}{\mathcal{K}}

\newcommand{\R}{\mathbb{R}}
\newcommand{\C}{\mathbb{C}}
\newcommand{\N}{\mathbb{N}}

\newcommand{\Z}{\mathbb{Z}}

\newcommand{\oo}{\Omega}

\newcommand{\cA}{\mathcal{A}}
\newcommand{\cB}{\mathcal{B}}

\newcommand{\cI}{\mathcal{I}}
\newcommand{\cK}{\mathcal{K}}
\newcommand{\cL}{\mathcal{L}}

\newcommand{\cR}{\mathcal{R}}
\newcommand{\cS}{\mathcal{S}}
\newcommand{\cX}{\mathcal{X}}
\newcommand{\cY}{\mathcal{Y}}

\newcommand{\pp}{\tfrac{\pi}{2}}

\title{A proof of Jones' conjecture}

\author{
	Jonathan Jaquette \thanks{Partially supported by NSF DMS 0915019,	NSF DMS 1248071} \thanks{
		Department of Mathematics, Hill Center-Busch Campus, Rutgers, The State University of New Jersey, Piscataway, NJ, USA, 08854-8019.
		{\tt jaquette@math.rutgers.edu}
	}
}
\begin{document}

\maketitle

\begin{abstract}
	In this paper, we prove that Wright's equation $y'(t) = - \alpha y(t-1) \{1 + y(t)\}$ has a unique slowly oscillating periodic solution for parameter values $\alpha \in (\tfrac{\pi}{2},  1.9]$, up to time translation. 
	This result proves Jones' Conjecture formulated in 1962, that there is a unique slowly oscillating periodic orbit for all $ \alpha > \tfrac{\pi}{2}$. 
	Furthermore, there are no isolas of periodic solutions to Wright's equation; all periodic orbits arise from Hopf bifurcations. 
\end{abstract}

\begin{center}
	{\bf \small Key words.} 
	{ \small Wright's Equation $\cdot$ Jones' Conjecture $\cdot$ Delay Differential
		Equations \\ Computer-Assisted Proofs $\cdot$ Branch and Bound $\cdot$ Krawczyk method} 
\end{center}
\tableofcontents

 \newpage
\section{Introduction}

An often studied class of delay differential equations are negative feedback systems of the form: 
\begin{equation}
x'(t) = - \alpha f(x(t-1)) 
\label{eq:MNF}
\end{equation}
where  $ xf(x) > 0 $ for $x \neq 0$ and $ f'(0) >0$.   
One particularly well studied example of \eqref{eq:MNF} is when $f(x) =e^x -1$, better known as Wright's equation, which after making the change of variables $ y =e^x -1 $ can be written in the following form:
\begin{equation}
y'(t) = - \alpha \,y(t-1) \left[  1+ y(t)  \right].
\label{eq:Wright}
\end{equation}

In \cite{jones1962existence}, Jones proved that for $\alpha > \pp$ there exists at least one  \emph{slowly oscillating periodic solution (SOPS)}.
That is, a periodic solution $y:\R \to \R$ which is positive for at least one unit of time (the delay time in Wright's equation), negative for at least one unit of time, and then repeats. 
In this paper we prove there is a unique SOPS to \eqref{eq:Wright} for $ \alpha \in ( \pp,1.9]$, thus completing a proof of Jones' conjecture: 
\begin{theorem}[Jones' conjecture]
	\label{prop:Jones}
	For every $ \alpha > \pp$ there exists a unique slowly oscillating periodic solution to \eqref{eq:Wright}. 
\end{theorem}

\begin{wrapfigure}{l}{0.5\textwidth}
	\centering
	\includegraphics[width =.5\textwidth]{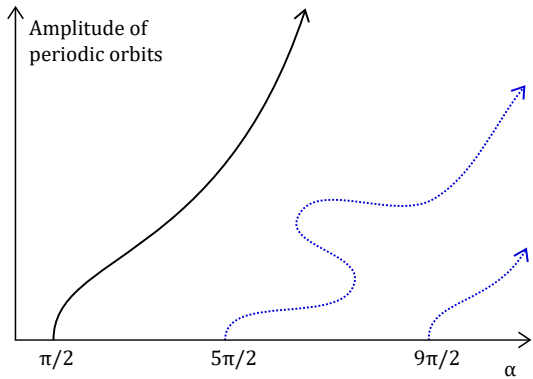}
		\caption{	\footnotesize 
			A bifurcation diagram for periodic solutions to Wright's equation. 
			There are no folds in the principal branch of slowly oscillating periodic solutions (solid curve).  While there may be  folds in  the branches of rapidly oscillating periodic solutions (dotted curves), it is conjectured that this does not occur. 
			There are no isolas of periodic solutions (not displayed).
		}
	
	\label{fig:BifurcationDiagram}
\end{wrapfigure}

This work contributes a capstone to many decades of mathematical work studying Wright's equation. 
To briefly review, a \emph{principal branch} of slowly oscillating periodic orbits is born at $ \alpha = \pp$ and continues on for all $ \alpha > \pp$ \cite{nussbaum1975global}. 
Moreover, Wright's equation has  supercritical Hopf bifurcations at $ \alpha = \pp + 2 n\pi$ for  integers $ n \geq 0$, with slowly oscillating periodic orbits arising when $n=0$, and rapidly oscillating periodic orbits arising when $n\geq1$ (see Figure \ref{fig:BifurcationDiagram}) \cite{chow1977integral}. 
Together with the parameter $\alpha$, the collection of periodic orbits forms a 2-dimensional manifold \cite{regala1989periodic}. 

A two-part geometric  version of Jones' conjecture was proposed in  \cite{lessard2010recent}: (i) the principal  branch of SOPS does not fold back on itself, and (ii)   there are no other connected components (\emph{isolas}) of SOPS. 
By   \cite{BergJaquette,lessard2010recent,jlm2016Floquet,xie1991thesis}   the principal branch does not have any folds $ \alpha > \pp$.

In \cite{jlm2016Floquet,xie1991thesis} it is shown that there is a unique SOPS for $ \alpha \geq 1.9$. 
These proofs use that fact that if every SOPS is asymptotically stable for some $ \alpha > \pp$, then there is a unique SOPS \cite{xie1993uniqueness}.   
Using  estimates describing SOPS  for when $\alpha$ is large \cite{nussbaum1982asymptotic}, Xie showed that there is a unique SOPS for all $ \alpha \geq 5.67$ \cite{xie1991thesis}. 
By using computer-assisted proofs to characterize SOPS to Wright's equation \cite{jlm2016Floquet},  this method was extended to show there is a unique SOPS for  $ \alpha \in [1.9,6.0]$.   

However, for $ \alpha$ close to the bifurcation value $ \pp$ the dynamics becomes center-like, and  proving uniqueness through these stability arguments becomes infeasible. 
To overcome this obstacle, we equate the problem of finding periodic orbits to \eqref{eq:Wright} with a zero--finding problem in a space of Fourier coefficients. 
We  then employ rigorous numerics to derive a computer-assisted proof that there is a unique SOPS to Wright's equation for $\alpha \in ( \pp, 1.9]$, thus proving the Jones conjecture.

Furthermore, 
Theorem \ref{prop:Jones} allows us to deduce that there are no isolas of rapidly oscillating periodic solutions. 
Since the nonlinearity in \eqref{eq:MNF} depends only on $x(t-1)$, in fact any periodic orbit is either a SOPS or rescaling thereof. 
This rescaling between slowly and rapidly oscillating periodic solutions is given in terms of a solution's lap number \cite{mallet1988morse} and its period, as detailed in the following theorem:

\begin{theorem}
	\label{prop:Rescaling}	 
	Let $ x_0$ be a periodic solution to \eqref{eq:MNF} at parameter $\alpha_0$ with period $L_0$ and lap number $ N $. 
	Then there exists a SOPS $x_1(t) = x_0(r t) $ to \eqref{eq:MNF} at parameter $\alpha_1= r\alpha_0$ where $r := 1- \tfrac{N-1}{2} L_0  $. 
\end{theorem}
Thus, every periodic orbit is on a branch originating from one of the Hopf bifurcations at $ \alpha = \pp + 2 n \pi$. 
That is to say, there are no isolas of rapidly oscillating periodic  solutions. 
However, this is not sufficient to show there are no folds in the branches of rapidly oscillating periodic solutions. 
The proofs for Theorem \ref{prop:Jones} and Theorem \ref{prop:Rescaling} are presented at the end of Section \ref{sec:GlobalAlgorithm}, and we discuss future directions in Section \ref{sec:FutureWork}.



\section{Outline of Proof}

In this paper we show that there is a unique
 slowly oscillating periodic orbit to \eqref{eq:Wright}  for all $ \alpha \in ( \pp, 1.9]$. 
Like in \cite{BergJaquette,lessard2010recent}, we recast the problem of studying the periodic orbits of \eqref{eq:Wright} as the problem of finding the zeros of a functional $F$ defined in a space of Fourier coefficients (see Section \ref{sec:FunctionDomain}).   
Since periodic solutions to \eqref{eq:Wright} must have a high degree of smoothness, 
in particular real analyticity  \cite{wright1955non,nussbaum-analytic},  
 their Fourier coefficients will decay very rapidly. 
That is to say, the functional we are interested in can  be well approximated by a Galerkin projection onto a finite number of Fourier modes.

In finite dimensions, there are efficacious techniques for rigorously locating and enumerating the solutions to a system of nonlinear equations by way of interval arithmetic \cite{neumaier1990interval,hansen2003global,moore2009introduction}. 
We apply these techniques in infinite dimensions, specifically the \emph{branch and bound} method, also referred to as a \emph{branch and prune} method.  
That is, we first construct  a bounded set $X$ of Fourier coefficients  which contains all the zeros of $F$ (see Section  \ref{sec:SolutionSpace}). 
Then we partition $X$ into a finite number of pieces $\{X_n \}$ which we   refer to as  \emph{cubes} (see Definition \ref{def:cube}).  
For each cube $X_n$ we are interested to know whether: 
\begin{enumerate}[$($a$)$]
	\item there exists  a unique point $\hat{x} \in X_n$ for which $ F(\hat{x})=0$, or
	\item there does not exist any points $\hat{x} \in X_n$ for which $F(\hat{x})=0$. 
\end{enumerate}
If we can show that $(a)$ holds for one cube, and $(b)$ holds for all the other cubes, then we will have shown that $ F=0$ has a unique solution.

This approach requires some additional preparation.  Since periodic orbits to \eqref{eq:Wright} form a 2-manifold in phase space \cite{regala1989periodic}, the  functional $F$ we construct in Section \ref{sec:FunctionDomain} will not have isolated zeros. 
The numerical techniques we employ are suited to finding isolated zeros, so  it is necessary to reduce the dimension of the kernel by two.  
Along the principal branch $\alpha$ can be taken as one of the coordinate dimensions.  We reduce this dimension by treating $\alpha$ as a parameter and performing our estimates uniformly in $\alpha$. 
The other dimension can be attributed to time translation; if $y(t)$ is a periodic orbit, then so is $ y(t+\tau)$ for any $\tau \in \R$.  
We reduce this dimension by imposing a phase condition; we may assume without loss of generality that the first Fourier coefficient is a positive real number (see Proposition \ref{prop:TimeTranslation}).

The central technique we use to determine whether $(a)$ or $(b)$ holds for a given cube is the Krawczyk method \cite{neumaier1990interval,moore2009introduction,hansen2003global,moore1977test}. 
For a function $f\in C^1(\R^n,\R^n)$ the Krawczyk operator takes as input a rectangular set $X\subseteq  \R^n$  and produces as output a rectangular set $K(X) \subseteq \R^n$. 
This set $K(X)$ has the properties that, 
$(i)$ if $K(X) \subseteq X$, then there exists a unique point $ \hat{x} \in X$ for which $f(\hat{x})=0$, and  
$(ii)$ if $\hat{x} \in X$ and $f(\hat{x}) =0$, then $ \hat{x} \in K(X)$. 
Clearly $(i) $ implies $(a)$, and if $ X \cap K(X) = \emptyset$ then $(b)$ follows. 
Additionally, even if we can prove neither $(a)$ nor $(b)$ our situation could still improve; we can replace $X \mapsto X \cap K(X)$ without losing any solutions.

Adjustments are needed to generalize the Krawczyk operator to infinite dimensional systems.  
In \cite{galias2007infinite} a Krawczyk operator is defined in Hilbert space to study fixed points and period-2 orbits in an infinite dimensional map.  
In Section \ref{sec:KrawczykBanach} we present a generalization of the Krawczyk operator to Banach spaces.

 \begin{wrapfigure}{r}{0.6\textwidth}
 	\centering
 	\includegraphics[width = .6\textwidth]{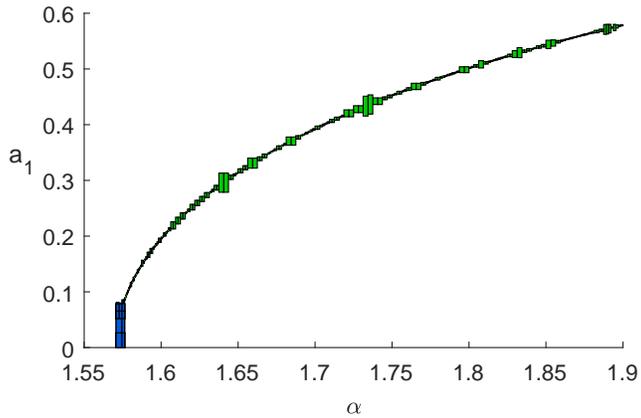}
 	\caption{ \footnotesize
 		The main result of this paper is a collection of ``cubes'' in Fourier space which cover the Fourier coefficients of SOPS to \eqref{eq:Wright}. 
 		The first Fourier coefficient of this cover is plotted here with respect to $\alpha$. 
 		Inside each green cube there exists a unique SOPS corresponding to each $\alpha$, essentially by Theorem \ref{prop:Krawczyk}. 
 		Inside each blue cube the only SOPS that can exist are on the principal branch, by \cite{BergJaquette}.
 	}
 	\label{fig:Verified}
 \end{wrapfigure}

To determine whether $(a)$ or $(b)$ holds the Krawczyk operator by itself is not always sufficient, and  we combine several additional tests to create a single \emph{pruning operator} (see Section \ref{sec:Prune}).  
One problem is that $y \equiv 0$ is always a trivial periodic solution to \eqref{eq:Wright}. 
To avoid this pitfall, we use Lemma \ref{prop:zeroneighborhood2} which rules out small periodic solutions \cite{BergJaquette}. 
A further difficulty is that at the Hopf bifurcation, the principal branch of periodic solutions is pinched to a point as their amplitudes approach zero.
To handle this case, we use Lemma \ref{prop:BifNbd} which explicitly gives a neighborhood about the Hopf bifurcation within which the only solutions that could exist are  on the principal branch \cite{BergJaquette}. 
Lastly, and most simply, if we can directly show   that $\|F\|$ is bounded away from zero on a cube $X_n$, then  $(b)$ holds.

Algorithm \ref{alg:BranchAndPrune} follows the standard format of a global branch and bound method. 
In short, for a collection of cubes we successively prune each of its cubes. 
If $(a)$ holds for a given cube, then it is set aside and added to a list of solutions. 
If $(b)$ holds for a given cube, then that cube is discarded. 
If the pruning operator significantly reduces the size of a cube, then the pruning operator is applied again. 
If none of these are the case, then the cube is split in half, and both pieces are added back to the collection of cubes to inspect. 
This process repeats until all of the cubes have been removed or reduced to a sufficiently small size.

The output of Algorithm \ref{alg:BranchAndPrune} is three collections of cubes:  $\cA$, $\cB$, and $\cR$ (see Figure \ref{fig:Verified}).
In Theorem \ref{prop:BnB} we show that these sets have the properties that,  $(i)$ each cube in $\cA$ has a unique solution with respect to $\alpha$, $(ii)$ the cubes in $\cB$ are near the Hopf bifurcation, with  any solutions contained therein  residing  on the principal branch, and $(iii)$ all solutions to $F=0$ are contained in $\bigcup \cA \cup \cB \cup \cR$.

Ideally $\cR = \emptyset$, and this will often be the case if the zeros of $F$ are simple and the algorithm is allowed to run a sufficiently long time. 
However we are trying to verify not just simple, isolated solutions, but a 1-parameter family of solutions.  
As such, sometimes when a cube is split in two this division will bisect the curve of solutions (see Figure \ref{fig:BranchANDBound}).
When this occurs the algorithm will be forced  to subdivide many cubes near where the solution curve was bisected, resulting in the variably sized cubes noticeable in Figure \ref{fig:Verified}. 
To address this we recombine the cubes in $\cR$ which have the same $\alpha$ values, then subsequently use the Krawczyk operator to show that $(a)$ holds on the recombined cubes (see Algorithm \ref{alg:Recombine}). 
In this fashion, we prove Theorem \ref{prop:Jones}.

\subsection{Krawczyk Operator} 
\label{sec:KrawczykBanach}

In numerical analysis  there are many variations on the theme of Newton's method:
\(
	x_{n+1} \mapsto x_n - Df(x_n)^{-1} f(x_n). 
\)
As inverting a matrix is computationally expensive, one  alternative method is to replace $ DF(x_n)^{-1}$ with a fixed matrix $A^\dagger \approx Df(x_0)^{-1}$. 
If $f(x_0) \approx 0$, then the Newton-Kantorovich theorem gives conditions for when the map $ T(x) = x - A^\dagger f(x)$  defines a contraction map in a neighborhood about $x_0$. 
The Krawczyk operator may be thought of as a way of bounding the image of $T$, itself being defined on rectangular sets $ X \subseteq \R^n$ and having the property that $T(X) \subseteq K(X,x_0)$.  
Rectangular, in the sense that $X$ can be given as the product of intervals in the  coordinate directions of $\R ^n$.  
Here we  generalize the Krawczyk operator to non-rectangular subsets of Banach spaces.  
\begin{definition}
	\label{def:Krawczyk}
	Let  $Y,Z$ denote Banach spaces and let $ A^\dagger:Z \to Y$ be an injective, bounded linear operator. Fix a convex, closed and bounded  set $X \subseteq Y$, a neighborhood  $ U \supseteq X$,  and a Frechet differentiable function $f:U \to Z $.
	Let  
	\[
	( I - A^\dagger Df(X))(X-\bar{x}) = \overline{conv} \left( \bigcup_{x_1,x_2 \in X} 	( I - A^\dagger Df(x_1)) (x_2-\bar{x})  \right),
	\]
	where $\overline{conv}$ denotes the closure of the convex hull.
	For a point $ \bar{x} \in X$ we define the Krawczyk operator $K(X,\bar{x})$ as:  
	\begin{equation}
	\label{eq:KrawczykDef}
	K(X,\bar{x}) := \bar{x} - A^\dagger f( \bar{x}) + ( I - A^\dagger Df(X))(X-\bar{x}) \subseteq Y.
	\end{equation}

\end{definition}
\noindent
Typically $ \bar{x}$ is taken to be the center of $ X$,  
and $A^\dagger$ is taken to be an approximate inverse of $ DF(\bar{x})$.
If $ K(X,\bar{x}) \subseteq X$ for a rectangular set $X \subseteq \R ^n$,   then there exists a unique $\hat{x}$ such that $ f(\hat{x})=0$. 
	In Theorem  \ref{prop:Krawczyk} we prove an analogous result.   
The existence of a fixed point is achieved by the Schauder fixed point theorem. 
However to prove uniqueness, dropping  the rectangular condition causes problems even  in finite dimensions; in Theorem 
 \ref{prop:Krawczyk} $(iv)$ we prescribe a hypothesis sufficient for proving uniquessness in our level of generality.

\begin{theorem}
			\label{prop:Krawczyk}
	Suppose $K$ is a Krawczyk operator as given in Definition \ref{def:Krawczyk} and $ T:= x - A^\dagger f(x)$. 
	\begin{enumerate}[(i)]
		\item  If $x \in X$, then $ T(x) \in K(X,\bar{x})$. 
		\item If $ \hat{x} \in X$ and $ f(\hat{x} )=0$, then $ \hat{x} \in K(X,\bar{x})$. 
		\item If $K(X,\bar{x}) \subseteq X$ and  $X$ is compact, then there exists a point $ \hat{x} \in X$ such that $ f(\hat{x})=0$. 
		\item If $K(X,\bar{x}) \subseteq X$ and there exists $0 \leq \lambda < 1$ such that $( I - A^\dagger Df(X))(X-\bar{x}) \subseteq \lambda \cdot ( X - \bar{x})$, then there exists a unique point $ \hat{x} \in X$ such that $ f(\hat{x})=0$. 
	\end{enumerate}
\end{theorem}

\begin{proof}

$\,$ 
	\begin{enumerate}[(i)]
		\item 
		Fix  a point $x \in X$ and write $ h = x - \bar{x}$. By the mean-value theorem for Frechet differentiable functions \cite{ambrosetti1995primer}, we have: 
		\begin{align*}
		T(x) 
		&= \bar{x} - A^\dagger f(\bar{x}) + \int_{0}^{1} DT( \bar{x} + t h)  \cdot h \, dt  \nonumber \\ 
		&=  \bar{x} - A^\dagger f(\bar{x}) + \lim_{N \to \infty}	\sum_{i=1}^N  \tfrac{1}{N} \left( I - A^\dagger Df(\bar{x} + \tfrac{i}{N}h) \right) \cdot h  \nonumber \\
		&\in   \bar{x} - A^\dagger f(\bar{x}) + \overline{conv} \left(  \left( I - A^\dagger Df(X) \right) \cdot (x - \bar{x}) \right)   \\
		&\subseteq K(X,\bar{x}). \label{eq:TsubK}
		\end{align*}

		\item  
		If there is some $ \hat{x} \in X$ such that $ f(\hat{x} )=0$, then $\hat{x} =  T(\hat{x}) \in K(X,\bar{x})$. 
		
		\item 	Since  $ T(X) \subseteq K(X,\bar{x})$ by \emph{(i)} and $K(X,\bar{x}) \subseteq X$ by assumption, therefore $ T(X) \subseteq X$.  
		As $T$ is continuous and $X$ is convex and  compact, then by the Schauder fixed point theorem there exists some $ \hat{x} \in X$ such that $\hat{x} =T(\hat{x})$. 
		 Since $A$ is injective,  the zeros of $f$ are in bijective correspondence   with the fixed points of $ T$, thereby $ f(\hat{x})=0$.

		\item 	Inductively define: 		$X_0=X$, $x_0 = \bar{x}$, and $X_{n+1} = T(X_n)$, $x_{n+1} = T(x_n) $. 
		Note that as $T(X) \subseteq X$ then $X_{n+1} \subseteq X_n$ for all $n$. 
		We show that $ X_n \subseteq x_n + \lambda^n (X_0 - x_0)$. 
		This is clearly true for $ n=0$. 
		For $ n \geq 1$ then:
		\begin{align*}
		X_{n+1} &\subseteq K(X_n,x_n) \\
		&=  x_n - A^\dagger f(x_n) + (I -A^\dagger Df(X_n)) \cdot (X_n -x_n) \\
		&\subseteq x_{n+1} + (I -A^\dagger D f(X_0)) \cdot \lambda^n (X_0 -x_0) \\
		&\subseteq x_{n+1} + \lambda^{n+1} (X_0 - x_0). 
		\end{align*}
		Since $ \lambda^{n}  \| X_0-x_0 \|$ can be made arbitrarily small and $ \{ x_n\}_{n=N }^\infty\subseteq X_N$, it follows that $ \{x_n\}$ is a Cauchy sequence. 
		As $X$ is complete, then $ \lim x_n = \hat{x} $ and additionally $  \bigcap_{n=0}^\infty X_n = \hat{x}$. 
		Thereby $\hat{x}$ is the unique fixed point  of $T$ in $X_0=X$ and the unique zero of $ f$ in $X$. 
		
	\end{enumerate}
\end{proof}

\subsection{Functions and Domains}
\label{sec:FunctionDomain}
As in \cite{BergJaquette,lessard2010recent}, we convert Wright's equation into a functional equation on the space of Fourier coefficients. 
For a continuous periodic function  $y:\R \to \R$  with frequency $\omega >0$, we may write it as: 
\begin{equation}
y(t)  = \sum_{k \in \Z} c_{k} e^{ i \omega k t }
\label{eq:FourierEquation}
\end{equation}
where $ c_k \in \C$ and $\sum_{k\in \Z} |c_k|^2 < \infty$.  
By \cite{BergJaquette}   it suffices to work with sequences $ \{ c_k \}_{k=1}^{\infty}$  to study periodic solutions to  \eqref{eq:Wright}.  
This is because real-valued functions have Fourier coefficients satisfying $ c_{-k} =c^*_k$, and  periodic solutions to \eqref{eq:Wright} necessarily satisfy  $c_0 =0$.  
Hence  we define the following Banach spaces: 
\begin{align}
\ell^1 :=& \left\{ \{ c_k \}_{k=1}^\infty : c_k \in \C \mbox{ and } \sum_{k=1}^\infty | c_k| < \infty  \right\} & \| c \|_{\ell^1} =& 2 \sum_{k = 1}^\infty | c_k| \\
\Omega^s :=& \left\{ \{ c_k \}_{k=1}^\infty : c_k \in \C \mbox{ and } \sup_{k \in\N} k^s |c_k| < \infty  \right\} & \| c \|_{s} =&  \sup_{k \in\N} k^s | c_k|.
\end{align}
The smoother a function is the faster its Fourier coefficients will decay; if a function is $s$--times continuously differentiable, then its Fourier coefficients will be in $ \oo^s$.  
Since periodic solutions to \eqref{eq:Wright} are real analytic \cite{wright1955non,nussbaum-analytic}, it follows that their Fourier coefficients will be in $ \oo^s$ for all $ s \geq 0$.

If $y$ is a solution to Wright's equation, then by substituting  \eqref{eq:FourierEquation} into \eqref{eq:Wright}  we obtain:
\begin{equation}
\sum_{k \in \Z} i \omega k c_k e^{ i \omega k t} = - \alpha \left(\sum_{k \in \Z}  c_k e^{-i \omega k}  e^{ i \omega k t}\right) \left( 1 + \sum_{k \in \Z}  c_k e^{ i \omega k t}\right).
\end{equation}
By matching the $ e^{ i \omega k t}$ terms, subtracting the RHS, and dividing through by $\alpha$, we obtain the following sequence of equations for $k \in \Z$ below: 
\begin{align}
[F(\alpha,\omega,c)]_k=&\,
\left( i \tfrac{\omega}{\alpha} k +  e^{ - i \omega k} \right) c_k +  \sum_{\substack{k_1,k_2\in\Z\\ k_1 + k_2 = k}} e^{- i \omega k_1} c_{k_1} c_{k_2} 
\label{eq:FourierSequenceEquation} \\
  =&\,
\left( i \tfrac{\omega}{\alpha}  k + e^{ - i \omega k }\right) c_k +   \sum_{j=1}^{k-1} e^{-i \omega j } c_j c_{k-j}  + \sum_{j=1}^{\infty} \left(e^{-i \omega (j+k) }  +e^{i \omega j } \right) c_j^* c_{j+k}.
\label{eq:FourierSequenceEquation_2} 
\end{align}
Dividing through by $\alpha$ ensures that the parameter dependence in $F$ is solely concentrated in the linear part. 
In this manner $y$ is a periodic solution with frequency $ \omega$  to Wright's equation at parameter $ \alpha$ if and only if $ [F(\alpha,\omega, c)]_k=0$ for all $ k \in \Z$ \cite{BergJaquette,jlm2016Floquet}.

To more succinctly express the functional  $F$ we introduce additional notation. 
For a sequence $ c=\{ c_k\}_{k=1}^\infty$ we denote the projection onto the $k$-coefficient by $ [c]_k := c_k$. 
We define unnormalized basis elements $ e_j \in \ell^1,\oo^s$ for $ j\in \N $  by:  
\begin{align*}
	[e_j]_k = 
	\begin{cases}
	1 & \mbox{ if } k = j, \\
	0 & \mbox{ if } k \neq j.
	\end{cases}
\end{align*}
We define the discrete convolution $a*b$ for $ a,b \in \ell^1$ component-wise by: 
\begin{align*}
	\left[ a * b \right]_k &:= \sum_{|k_1| + |k_2| = k} a_{k_1} b_{k_2}  & 
	& = \sum_{j=1}^{k-1} a_j b_{k-j} + \sum_{j=1}^\infty a_j^* b_{k+j} + a_{k+j} b_j^*,
\end{align*}
where $  a_{-k} = a_k^*$ and $ b_{-k}=b_k^*$, and the sum is taken over $ k_1 , k_2 \in \Z$.  
The space $\ell^1$ is a Banach algebra,  which is to say that $\| a * b \|_{\ell^1} \leq \|a \|_{\ell^1}   \| b\|_{\ell^1} $ for all $ a , b \in \ell^1$. 
While $ \oo^s$ is not a Banach algebra \emph{per se}, if $ s \geq 2$ then there exists a constant $B \geq 0$ such that  $ \| a * b \|_s \leq B \| a \|_s   \| b \|_s$  for all $a ,b \in \Omega^s$ (see \cite{lessard2010recent,berg2008chaotic}).  
Lastly, we define a linear operator $\K: \Omega^{s} \to \Omega^{s+1}$ and a continuous family of linear operators  $U_\omega :  \Omega^{s }\to \Omega^{s-1}$ as below: 
\begin{align*}
[\K c ]_k &:=  c_k  /k ,&&&
[ U_\omega c ]_k &:= e^{-i k \omega} c_k .
\end{align*}
The loss of regularity in the range of $U_\omega$ is necessary for its continuity, as $\frac{\partial}{\partial \omega} U_{\omega} = - i \K^{-1} U_{\omega}$. 
We may extend $ U_\omega$ to act on bi-infinite sequences $\{c_k\}_{k \in \Z}$ using the same component-wise definition. Additionally, this extension is compatible with our definition of the discrete convolution, as $[U_\omega c]_k^* = [U_\omega c]_{-k}$ whenever $c_{k}^* = c_{-k}$. 
In Definition \ref{def:Functional} we rewrite \eqref{eq:FourierSequenceEquation} in operator notation and list several propositions, the proofs of which are left to the reader.  

\begin{definition}
	\label{def:Functional}
	Define the function $F:\R^2 \times  \Omega^{s} \to  \Omega^{s-1}$ as: 
	\begin{equation}
	F(\alpha,\omega,c) := ( i \tfrac{\omega}{\alpha} \K^{-1} +  U_{\omega}) c +  (U_{\omega} c) * c.
	\end{equation}
	
\end{definition}
\begin{proposition}[Theorem 2.2 in \cite{BergJaquette}]
	\label{prop:Equivalence}
	Let $\alpha, \omega>0$.
	If $ c \in \ell^1$ solves
	$F(\alpha,\omega,c) =0$,
	then $y(t)$, given by~\eqref{eq:FourierEquation} with $ c_0=0$ and $ c_{-k} = c^*_k$, is a periodic solution of~\eqref{eq:Wright} with period~$2\pi/\omega$.
	Vice versa, if $y(t)$ is a periodic solution of~\eqref{eq:Wright} with period~$2\pi/\omega$, then its Fourier coefficients  satisfy $c_0 = 0 $, $c_{-k} =c_{k}^*$,  $\{c_k\}_{k=1}^\infty \in \ell^1 $ and solve $F(\alpha,\omega,\{c_k\}_{k=1}^\infty) =0$. 
\end{proposition}

\begin{proposition}
	\label{prop:Frechet}
	For each $\alpha > 0$ and $ s \geq 2 $ the function  $F:\R^2 \times  \Omega^{s} \to    \Omega^{s-1}$ is Frechet differentiable, with partial derivatives given as: 
	\begin{align}
	\frac{\partial}{\partial \omega} 
	F(\alpha,\omega,c) &= i \K^{-1} (\alpha^{-1}  I - U_{\omega}) c - i  (\K^{-1} U_{\omega} c) * c  \label{eq:dFdW} \\
	\frac{\partial }{\partial c} F(\alpha,\omega,c) \cdot h &= ( i \tfrac{\omega}{\alpha} \K^{-1} + U_\omega) h + (U_\omega c) * h + (U_\omega h )*c, \label{eq:dFdC} 
	\end{align}
	where $ h \in \Omega^{s}$. 
\end{proposition}

\begin{proposition}
	\label{prop:DerivativeComponent}
	Define $\gamma_1(k,n) := e^{-i \omega(n+k)} + e^{i \omega n} $ and $ \gamma_2(k,n) := e^{-i \omega n} + e^{i \omega (n-k) } $. 
	Writing $ c_k = a_k + i b_k$, the component-wise derivatives of $ F$ are given as:
	\begin{align*} 
	\frac{\partial}{\partial \omega }
	[F(\alpha, \omega,c)]_k 
	=& \;
	i k ( \alpha^{-1}  -    e^{ - i \omega k }) c_k 
	-i   \sum_{j=1}^{k-1} j e^{-i \omega j } c_j c_{k-j}  \nonumber \\ &
	\;- i  \sum_{j=1}^{\infty} \left( (j+k)e^{-i \omega (j+k) }  -j e^{i \omega j } \right) c_j^* c_{j+k}   . \\
	\frac{\partial}{\partial a_n} [F(\alpha,\omega,c)]_k 
	=& \;
	( i \tfrac{\omega}{\alpha}  k + e^{ - i \omega k })+
	\begin{cases}
	\gamma_1 c_{n+k} + \gamma_2  c_{k-n}  & \mbox{if } 1\leq n < k \\
	\gamma_1 c_{n+k} + \gamma_2  c_{n-k}^* 
	& \mbox{if } k \leq n . 
	\end{cases} \\
	\frac{1}{i}	\frac{\partial}{\partial b_n} [F(\alpha,\omega,c)]_k 
	=& \;
	( i \tfrac{\omega}{\alpha}  k + e^{ - i \omega k })+
	\begin{cases}
	-\gamma_1  c_{n+k} + \gamma_2  c_{k-n}  & \mbox{if } 1 \leq n < k \\
	-\gamma_1 c_{n+k} + \gamma_2  c_{n-k}^* 
	& \mbox{if } k \leq  n. 
	\end{cases} 
	\end{align*}
\end{proposition}

\subsection{Decomposition of Phase Space}

By working in a space of rapidly decaying Fourier coefficients, we are able to closely approximate the value of $F$ using a Galerkin projection. 
Since $F: \R^2 \times  \oo^{s} \to \oo^{s-1}$ has distinct domain and range, we need to define two sets of projection maps. 
We define projection maps $ \pi_\alpha ,\pi_{\omega} : \R^2 \times \oo^s  \to \R $ and $ \pi_c : \R^2 \times \oo^s \to \oo^s$ on points $ x= ( \tilde{\alpha},\tilde{\omega},\tilde{ c}) \in \R^2 \times \oo^s$ as:
\begin{align}
	\pi_{\alpha} (x) &:= \tilde{\alpha} 
	&
	\pi_{\omega} (x) &:= \tilde{\omega}
	&
	\pi_{c} (x) &:= \tilde{c} .
\end{align}
For a fixed  integer $M \in \N$, define the projection maps  $\pi_{M}, \pi_{\infty} : \oo^s \to \oo^s$ by: 
\begin{align}
\pi_{M}(c) &:= \sum_{k=1}^M [c]_k e_k &&&
\pi_{\infty}(c) &:= c - \pi_{M} (c).
\label{eq:Galerkin1}
\end{align}
Define the  projection maps $\pi_{M}' ,\pi_\infty': \R^2 \times \oo^s \to \R^2 \times \oo^s$ by: 
\begin{align}
\pi_{M}'(c) &:=  (\pi_\alpha(x), \pi_\omega(x), \pi_{M} \circ \pi_c (x)) &&&
\pi_{\infty}'(c) &:= (0,0, \pi_{\infty} \circ \pi_c (x)).
\label{eq:Galerkin2}
\end{align}
For any bounded set $X \subseteq \R^2 \times \oo^s$, define: 
\[
|X|_k := \sup_{x \in X} \left| [\pi_{c} (x)]_k \right|.
\] 
We   define for $ F$ its Galerkin projection and remainder $ F_M, F_\infty : \R^2 \times \oo^{s} \to \oo^{s-1}$ as follows: 
\begin{align}
F_M( x) &:= \pi_M \circ F( \pi_M'(x)),  &
F_\infty(x) &:= F(x) - F_M(x).
\end{align} 
By construction  $ F = F_M + F_\infty$.

To show that there is a unique SOPS to \eqref{eq:Wright} we need to evaluate $F$ not just on single points but on voluminous subsets of its domain. 
The central subset of  $\R^2 \times \oo^s$ we consider in this paper  are \emph{cubes} which we define as follows: 
\begin{definition}
	\label{def:cube}
	For $M \in \N$, $s\geq 0$, $C_0>0$   define a cube $ X := X_M \times X_\infty \subseteq \R^2 \times \Omega^s$ to be of the following form: 
	\begin{align}
	X_{M} 		&:= [\underline{\alpha},\overline{\alpha}] \times [\underline{\omega},\overline{\omega}]   \times \prod_{k=1}^M [\underline{A}_k , \overline{A}_k] \times [\underline{B}_k , \overline{B}_k]  \label{eq:XMdef} \\	
	X_\infty 	&:= \left\{ c_k  \in \C : |c_k| \leq C_0 /k^s  \right\}_{k=M+1}^\infty \label{eq:XIdef}. 
	\end{align}		
\end{definition}
To  denote the union of  a collection of cubes $\cS: = \{  X_i \subseteq \R ^2 \times \oo^s \}$  we define 
	$	 \bigcup \cS := \bigcup_{X \in \cS} X 		 \subseteq \R ^2 \times \tilde{\oo}^s$.

There are primarily two reasons we have chosen to consider cubical  subsets of $ \R^2 \times \oo^s$. 
Firstly,  cubes are particularly easy to refine into smaller pieces. 
This is useful because to begin using a branch and bound method, we need to obtain global bounds on the solution space, and then partition these bounds into smaller pieces.
In practice, we  reduce the size of a cube by either subdividing it along a lower dimension into two cubes, or replacing the cube by its intersection with the Krawczyk operator: $X \mapsto X \cap K(X,\bar{x})$. 
In both these cases the resulting object is again a cube. 
In this manner, we can use cubes to cover the solutions to $ F=0$, and then refine the cover  using successively smaller cubes.

Secondly, cubes facilitate  explicit computations of $F_M$ and analytical estimates of $F_\infty$. 
While formally $F_M$ is an infinite dimensional map, computationally, we may consider $F_M$ to be a map $ \R^2 \times \C^M \to \C^M$. 
To calculate $F_M$, we simply truncate the second sum in \eqref{eq:FourierSequenceEquation_2} at $ j = M-k$. 
As the $\pi'_M$ projection of a cube is given as a finite product of intervals, it is well suited for using interval arithmetic \cite{moore2009introduction} to bound the image of $F_M(X)$. 
On the other hand, bounding  $F_\infty$ requires significantly more analysis. 
Below is a simple, yet ever recurring estimate in our calculations:  
\begin{equation}
\label{eq:SumIntegral}
	 \sum_{k=M+1}^\infty \frac{1}{k^s} \leq \int_M^\infty \frac{1}{x^s} dx = \frac{1}{(s-1)M^{s-1}},
\end{equation}
where we take $ s >1$. 
For example, if a cube $ X \subseteq \R ^2 \times \oo^s$ satisfies $ s >1$, then $ \| \pi_c x \|_{\ell^1} \leq   2 \sum_{k=1}^M |X|_k  + \frac{2 C_0 }{(s-1)M^{s-1}} $ for all $ x \in X$. 
This specific bound on the $\ell^1$ norm is later used  in Algorithm \ref{alg:Prune} to check whether Lemmas \ref{prop:zeroneighborhood2} or  \ref{prop:BifNbd} apply.

\begin{lemma}[Theorems E.1 and E.2 in \cite{BergJaquette}]
	\label{prop:zeroneighborhood2}
	Let $\omega \geq 1.1$, $\alpha \in (0,2]$, and define 
	\begin{equation}
	g(\alpha,\omega) := \sqrt{ \left(1- \tfrac{\omega}{\alpha} \right)^2 + 2   \, \tfrac{\omega}{\alpha} \left( 1- \sin  \omega \right)} .
	\label{eq:ZeroNbd}
	\end{equation}
	If $F(\alpha,\omega,c)=0$, then either $c \equiv 0$ or $	 g(\alpha,\omega) \leq  \|c \|_{\ell^1}   $.
\end{lemma}
\begin{lemma}[Theorem 4.10 \cite{BergJaquette}]
	\label{prop:BifNbd}
	For each $\alpha  \in  (\pp , \pp + 0.00553] $ there is  a unique   (up to time translation) periodic solution to Wright's equation with Fourier coefficients satisfying $  \|c \|_{\ell^1} \leq  0.18$ and having frequency $ | \omega - \pp | \leq 0.0924$.  
\end{lemma}

We note that while Lemma \ref{prop:zeroneighborhood2} is stated only for $ \omega \geq 1.1$ and $ \alpha \in (0,2]$, a more general formula is given in \cite{BergJaquette}. 
Also, we present  the hypothesis of Lemma \ref{prop:BifNbd} in terms of a bound on $ \| c\|_{\ell^1}$ as opposed to a bound on $ \| y'\|_{L^2}$ as in the original paper. 
This allows us to use the stronger result derived in the proof of   \cite[Theorem 4.10]{BergJaquette}, namely that the solution \emph{exists} and is unique, as opposed to the exact result stated in  \cite[Theorem 4.10]{BergJaquette}, which is that there is most one periodic solution.  

The remainder of this section is dedicated to proving Lemma \ref{prop:CentralLemma}, which  estimates $F_\infty$, its derivatives, and convolution products resulting from points inside of a cube.  
These estimates are used in Definition \ref{def:KrawczykApprox} to construct an outer approximation to the Krawczyk operator. 
 The reader is encouraged to skip 
 the proof of Lemma \ref{prop:CentralLemma} 
  on a first reading, 
  which is best  summarized as bounding various infinite sums by  various finite sums and the estimate in \eqref{eq:SumIntegral}.
These bounds are presented in Definition \ref{def:Gfunctions}, all of which are  given as a finite number of operations, explicitly computable in terms of $C_0$ and the $\pi'_M$-projection of a given cube. 
In Lemma \ref{prop:DiscreteConv} we define the constant $\gamma_M$ which is needed for the definition of \eqref{def:giiC}.

\begin{lemma}[Lemma 24 \cite{berg2008chaotic}]
	\label{prop:DiscreteConv}
	Let  $s \geq 2$ and let $s_*$ be the largest integer such that $ s_* \leq s$ and define: 
	\[
	\gamma_k := 2 \left[ \frac{k}{k-1}\right]^s + \left[ \frac{4 \ln (k-2)}{k} + \frac{\pi^2 -6}{3} \right] \left[ \frac{2}{k} + \frac{1}{2} \right]^{s_*-2}.
	\]
	For $k \geq 4$, we have that $	\sum_{k_1 =1}^{k-1} \frac{k^s}{k_1^{s} (k-k_1)^s} \leq  \gamma_k $. 
	If $ 6 \leq M \leq k$,  then $ \gamma_k \leq  \gamma_M $. 
\end{lemma}

\begin{definition}
	\label{def:Gfunctions}
	Fix a cube $ X  $ with $ s >2$, 
	define $C_1 := \sup_{x\in X} \| \pi_c x \|_s$, 
	 and select a point 	
	$ \bar{x} = (\bar{\alpha},\bar{\omega} , \bar{c}) \in X$  such that $ \bar{x} = \pi'_M(\bar{x} )$.  
	Define $ H = X - \bar{x}$, and define $ \Delta_\omega \in \R $ such that $\Delta_\omega \geq \sup_{x \in H} |\pi_\omega (x) - \bar{\omega}| $.   
	
	Define $h,g_M^{i},g_M^{ii} $ to be functions of the form 
	$ g_M : X \mapsto g_M(X) \in \R^M$ 
	and define $g_\infty^i,g_\infty^{ii,a},g_\infty^{ii,b}$ to   be functions of the form $ g_\infty: X \mapsto g_\infty(X) \in \R$ as follows:
	
	\begin{align}
	[h(X)]_k := &\; \frac{2 C_0^2}{(s-1)M^{s-1} (M+k+1)^{s} }+  2 C_0 \sum_{j = M-k+1}^M  \frac{ |X|_j}{(j+k)^s} \label{eq:F_tail} \\
		[g_M^{i}(X)]_k  := &\;
	2 C_0  \Delta_\omega \sum_{j = M-k+1}^M  \frac{|X|_j}{(j+k)^{(s-1)}}    \nonumber \\
	&+ \frac{C_0^2  \Delta_\omega}{(s-2)(M+k+1)^sM^{(s-2)} }   + 
	\frac{C_0^2  \Delta_\omega }{(s-1)(M+k+1)^{(s-1)}M^{(s-1)} }   \label{eq:giM}  \\
	[g_M^{ii}(X)]_k  := & \;
	\frac{4 C_0^2 }{(s-1)(M+k+1)^s M^{s-1}}    + 
	2 C_0 \sum_{j = M-k+1}^M  \frac{|H|_j}{(j+k)^s}   \label{eq:giiM}  \\
	g_\infty^i(X) := &\;
    \max_{ M+1 \leq k \leq 2M } k^{s} \sum_{j=k-M}^{M}   |\bar{c}_j \bar{c}_{k-j}|    \label{eq:giInfty}\\
	g_\infty^{ii,a}(X) := &\; \max_{M+1 \leq k \leq 2 M}  k^s  \sum_{j=k-m}^M |H|_j |X|_{k-j}   \nonumber  \\
	&+ \frac{2 C_0^2 (2^s+1)}{(s-1)M^{s-1}} + C_0    \sum_{j=1}^{M} \left(|X|_j +|H|_j \right) \left( \left( \frac{M+j+1}{M+1} \right)^s +1 \right)    
	\label{def:giiB}  	\\
	g_\infty^{ii,b}(X) := &\;  \frac{C_1^2 \gamma_{M+1}}{2} +    C_ 0C_1 \left( \frac{s-1}{(M+2)(s-2)} + \frac{s}{s-1} \right)  . 		\label{def:giiC} 
	\end{align}

\end{definition}

\begin{lemma} 
	\label{prop:CentralLemma}
	Fix a cube  $X$ with $ M \geq 5$, $s >2$, a point $ \bar{x} \in X$  such that $ \bar{x} = \pi'_M(\bar{x} )$, and define $H = X - \bar{x}$.  
	Then the following inequalities hold: 
	\begin{align}
	\label{eq:prop_F_tail} 
	\sup_{x \in X} 
	\left|F_{\infty}(x)  \right|_k &< [h(X)]_k
	&  1 \leq  k \leq M	\\
	\label{eq:prop_finite_defect_w}
	\sup_{x \in X , h \in H} 
	\left|\tfrac{\partial}{\partial \omega} F_{\infty}(x)
	 \cdot \pi_\omega(h) \right|_k 
	&\leq [g_M^{i}(X)  ]_k
	&  1 \leq  k \leq M \\
	\label{eq:prop_finite_defect_c}
		\sup_{x \in X , h \in H} 
	\left|\tfrac{\partial}{\partial c} F_{\infty}(x) 
	\cdot \pi_c(h)  \right|_k 
	& \leq[ g_M^{ii}(X) ]_k 
	& 1 \leq  k \leq M  \\
	\label{eq:prop_F_infty_center}
	\left| 		F_\infty(  \bar{x})  \right|_k &\leq	\frac{1}{ k^s} g_\infty^i(X) 		
	& M + 1 \leq  k \\
	\label{eq:prop_giiB}
	\sup_{x \in X , h \in H} 
	\left|\pi_c(h)* \pi_c(x) \right|_k  &\leq \frac{1}{k^s}  g_\infty^{ii,a}(X) 
	& M + 1 \leq  k \\
	\label{eq:prop_giiC}
		\sup_{x_1  , x_2 \in X} 
	\left|(\K^{-1} \pi_c(x_1)) * \pi_c(x_2) \right|_k &\leq \frac{1}{k^{s-1}} g_\infty^{ii,b}(X) 
	& M + 1 \leq  k .
	\end{align}
	
\end{lemma}

Throughout, let us write   $X_M = \pi'_M (X)$,    $H_M = \pi'_M (H)$, and $H_\infty = \pi_\infty' (H)$, noting also that $ H_\infty = \pi_\infty'(X)$.

\begin{proof}[Proof of \eqref{eq:prop_F_tail}]
	We show that $\left|F_{\infty}(x)  \right|_k  < [h(X)]_k$ for $   1 \leq  k \leq M$ and all $x \in X$. 
	Fix $x =  (\alpha,\omega,c)  \in X$, and write $c_M = \pi_M(c) $ and $c_\infty = \pi_\infty (c)$. 
	We compute: 
	\begin{align*}
	\pi_M \circ F_\infty (x) &= \pi_M \circ  \left(F(x) - F( \pi'_M x) \right)\\
	&= \pi_M \circ   
	\left(
	( U_\omega  c ) * c - ( U_\omega c_M) * c_M 
	\right)
	\\
	&= 
	\pi_M \circ   
	\left(   ( U_\omega c_M) * c_\infty   +  (  U_\omega c_\infty) * c_M + (U_\omega c_\infty) * c_\infty 
	\right)
	\end{align*}
	Since $ |U_\omega c |_k = |c|_k$, it follows that for $1 \leq k \leq M$ we compute the estimate  below:  
	\begin{align*}
	\left|   ( U_\omega c_M) * c_\infty  \right|_k  + \left| (U_\omega c_\infty )* c_M  \right|_k 
	\leq& 2 \sum_{j=1}^\infty 
	|c_M^*|_j  |c_\infty|_{k+j} + |c_M|_{k+j}  |c_\infty^*|_{j} \\
	=& 2  
	\sum_{j = M-k+1}^M   \left| c_M^*   \right|_{j}  \left| c_\infty \right|_{j+k} \\
	\leq& 2  \sum_{j = M-k+1}^M  | X|_j  \frac{C_0}{(j+k)^s} .
	\end{align*} 
	The last estimate uses the property that  $ |c_j| \leq C_0/j^s$ for $ j \geq M+1$.

	We calculate $( U_\omega  c_\infty) * c_\infty$  as below, again  using $ |c_j| \leq C_0/j^s$ for $ j \geq M+1$. 
	\begin{align*}
	\left| ( U_\omega  c_\infty) * c_\infty  \right|_k 
	&\leq  \sum_{j=M+1}^\infty  |c^*_\infty|_j |c_\infty |_{k+j} + | c_\infty |_{j+k} |c_\infty^*|_j  \\
	&\leq  \sum_{j=M+1}^\infty \frac{2 C_0^2}{j^s (j+k)^s}  
	\leq \frac{2 C_0^2}{(s-1)M^{s-1} (M+k+1)^{s} }.
	\end{align*}
	Hence for $ 1 \leq k \leq M$, it follows that: 
	\begin{align*}
		\left|F_{\infty}(x)  \right|_k  &\leq \; \frac{2 C_0^2}{(s-1)M^{s-1} (M+k+1)^{s} }+  2 C_0 \sum_{j = M-k+1}^M  \frac{ |X|_j}{(j+k)^s}  \\
		&= [h(X)]_k .
	\end{align*}

\end{proof}

\begin{proof}[Proof of \eqref{eq:prop_finite_defect_w}]
	We show that $\left|\tfrac{\partial}{\partial \omega} F_{\infty}(x) \cdot \pi_\omega(h) \right|_k 
	\leq [g_M^{i}(X)  ]_k $ for $
	1 \leq  k \leq M  $ and all $ x \in X$ and $h \in H$.		
	Select some $x = (\alpha , \omega , c) \in X$  and write $c_M = \pi_M(c) $ and $c_\infty = \pi_\infty (c)$.  
	From \eqref{eq:dFdW} we can calculate  $\frac{\partial}{\partial \omega}  
	F_\infty(x)   $ as follows: 
	\begin{align}
	\tfrac{\partial}{\partial \omega}   	F_\infty (x) 
	 &= - i  (\K^{-1} U_{\omega} c) * c  +   i  \pi_M (\K^{-1} U_{\omega} c_M) * c_M  
	 \nonumber \\
	&= -i \pi_{\infty}\left( \cK^{-1} U_\omega c_M \right) *c_M 
	- i \left( \cK^{-1} U_\omega c_M \right) * c_\infty  
	- i \left( \cK^{-1} U_{\omega} c_\infty \right) ( c_M + c_\infty) .
	\nonumber 
	\end{align}
	Hence, for $ 1 \leq k \leq M$ we may calculate the following: 
	\begin{align}
	\left| \tfrac{\partial}{\partial \omega}   F_\infty(x)   \right|_k& 
	\leq    \sup_{c_M \in X_M; \,c_\infty,c_\infty' \in H_\infty}
	\left| ( \K^{-1}   c_M) * c_\infty   
	+  ( \K^{-1}   c_\infty) * c_M + ( \K^{-1}   c_\infty) * c_\infty' \right|_k. \label{eq:dF_infty_dW}
	\end{align}

	For $1 \leq k \leq M$ and any $c_M \in X_M,c_\infty \in H_\infty$ we can simplify the first two summands in \eqref{eq:dF_infty_dW} as follows: 
	\begin{align*}
		 (\K^{-1}   c_M) *_k c_\infty  
		&= 
		\sum_{j =1}^\infty  [\cK^{-1} c_M^*]_j [c_\infty]_{k+j} + [\cK^{-1} c_M]_{k+j} [c_\infty^*]_{j} &
		&= 
		\sum_{j =M+1-k}^\infty j [c_M^*]_j [c_\infty]_{k+j} \\
		 (\K^{-1}   c_\infty) *_k c_M 
		&= 
		\sum_{j =1}^\infty  [\cK^{-1} c_\infty^*]_j [c_M]_{k+j} + [\cK^{-1} c_\infty]_{k+j} [c_M^*]_{j} & 
		&= 
		\sum_{j =M+1-k}^\infty  (k+j) [c_\infty]_{k+j} 	[c_M^*]_j	.
	\end{align*}
	Hence, we have the following estimate: 	
	\begin{align}
	 (\K^{-1}   c_M) *_k c_\infty  +  (\K^{-1}   c_\infty) *_k c_M 
	  &= 
	\sum_{j = M-k+1}^M   (2j+k)[c_\infty]_{j+k} [c_M^*]_j 
	\nonumber
	 \\
\left| (\K^{-1}   c_M) * c_\infty\right|_k  +  \left|(\K^{-1}   c_\infty) * c_M  \right|_k	&\leq  \sum_{j = M-k+1}^M  \frac{(2j+k)C_0}{(j+k)^s} |X|_j
	\nonumber
 \\
	&\leq 2 C_0 \sum_{j = M-k+1}^M  \frac{|X|_j}{(j+k)^{s-1}}  .
	\label{eq:dF_infty_dW_A}
	\end{align}
	Again, we used the estimate  $ |c_j| \leq C_0/j^s$ for $ j \geq M+1$. 
	We estimate the third summand in \eqref{eq:dF_infty_dW} for $c_\infty ,c_\infty' \in H_\infty$ 
	as follows: 
	\begin{align}
	 (\K^{-1}   c_\infty) *_k c_\infty'  &=  \sum_{j=M+1}^\infty j [c_\infty^*]_j [c_\infty']_{k+j} + (j+k) [c_\infty]_{j+k} [c_\infty'{}^*]_j 
	 	\nonumber \\
	\left| (\K^{-1}   c_\infty) * c_\infty'  \right|_k &\leq  \sum_{j=M+1}^\infty \frac{C_0^2}{j^{(s-1)} (j+k)^s} + \frac{C_0^2}{j^{s} (j+k)^{(s-1)}} 
		\nonumber \\
	&\leq \frac{C_0^2 }{(s-2)(M+k+1)^sM^{(s-2)} }   + 
	\frac{C_0^2 }{(s-1)(M+k+1)^{(s-1)}M^{(s-1)} }   .
		\label{eq:dF_infty_dW_B}
	\end{align}
	By combining the estimates from \eqref{eq:dF_infty_dW_A} and \eqref{eq:dF_infty_dW_B} into \eqref{eq:dF_infty_dW}, and recalling our choice of $\Delta_\omega$ in Definition \ref{def:Gfunctions}, then for $	1 \leq  k \leq M  $  we obtain the following:
	\begin{align*}
\sup_{x \in X, h \in H}		\left|\tfrac{\partial}{\partial \omega} F_{\infty}(x) \cdot \pi_\omega(h) \right|_k 
		\leq    &\;
		2 C_0 \Delta_\omega \sum_{j = M-k+1}^M  \frac{|X|_j}{(j+k)^{(s-1)}}    + \frac{C_0^2 \Delta_\omega }{(s-2)(M+k+1)^sM^{(s-2)} }  \nonumber \\
		&  + 
		\frac{C_0^2 \Delta_\omega }{(s-1)(M+k+1)^{(s-1)}M^{(s-1)} }    \\
		=&\;[g_M^{i}(X)  ]_k .		
	\end{align*}
	
\end{proof}

\begin{proof}[Proof of \eqref{eq:prop_finite_defect_c}]
	We show that $ \left|\tfrac{\partial}{\partial c} F_{\infty}(x) \cdot \pi_c(h)  \right|_k 
	\leq[ g_M^{ii}(X) ]_k $ for $
	1 \leq  k \leq M   $ and all $ x \in X$ and $h \in H$. 
	Let $(\alpha , \omega , c) \in X$ and  $ h \in \pi_c(H)$.  
From \eqref{eq:dFdC} we calculate $ \tfrac{\partial}{\partial c} (F(X)-F_M(X)) \cdot h $ below:  
	\begin{align*}
	\tfrac{\partial}{\partial c} (F(x)-F(\pi_M'x)) \cdot h  =&  \left(  (U_\omega h) * c + (U_\omega c)*h \right)  - \left( (U_\omega h) * c_M + (U_\omega c_M)*h \right)  \\
	=&   (U_\omega h) * (c-c_M) + (U_\omega (c-c_M) )*h  .
	\end{align*}
	Since $ c-c_M \in H_\infty$, it follows that: 
	\begin{equation*}
		| \tfrac{\partial}{\partial c} [F(x)-F(\pi'_Mx)] \cdot h |_k 
		\leq  \sup_{h \in H , h' \in H_{\infty}} 2 \cdot |h  * h'_{\infty} |_k.
	\end{equation*}
For $h \in H$ and $ h' \in H_\infty$  and 	for $ 1 \leq k \leq M$, we calculate $h*_k h'$ below, using the property that  $[h']_j=0$ for $j \leq M$. 
	\begin{align*}
	 h*_k h' &=  
	\sum_{j=1}^\infty \, [h^*]_j  [h']_{k+j} + [h]_{k+j} [h'{}^{*}]_j \\
	&= 	\sum_{j=M-k+1}^{M}  [h^*]_{j} [h']_{ k+j} + 
	 \sum_{j=M+1}^{\infty }  [h^*]_{j} [h']_{k+j} + [h]_{k+j} [h'{}^{*}]_j  .
	\end{align*}
	By  applying the estimates  $|h_j| \leq |H|_j$ for $j \leq M$, and  $ |h|_j,|h'|_j \leq C_0 / j^s$ for $ j \geq M+1$, we  obtain the following: 
	\begin{align*}
	\left|	  \tfrac{\partial}{\partial c } F_\infty(x) \cdot h \right|_k 	
	&\leq 2
	\left(
	 \sum_{j=M-k+1}^{M} |H|_j \frac{C_0}{(j+k)^s} + 
	\sum_{j=M+1}^{\infty } \frac{2 C_0^2 }{j^s (j+k)^s}  \right)\\
	&\leq 
	2C_0  	\sum_{j=M-k+1}^{M}  \frac{|H|_{j}}{(j+k)^s}  \; +\; \frac{4 C_0^2 }{(s-1)(M+k+1)^s M^{s-1}} \\
		&= [g_M^{ii}(X)]_k.
	\end{align*}
\end{proof}

\begin{proof}[Proof of \eqref{eq:prop_F_infty_center}]
	We show that $\left| 		F_\infty( \bar{\alpha}, \bar{\omega} , \bar{c})  \right|_k  \leq	\frac{1}{ k^s} g_\infty^i(X) 	$ for $	
	M + 1 \leq  k  $. 	
	Since $ \pi'_M (\bar{x}) = \bar{x}$ and $  [ \bar{c}]_k =0$ for $ k \geq M+1$, it follows that:
	\begin{equation}
	[	F_\infty ( \bar{\alpha}, \bar{\omega} , \bar{c})]_k = 
	\begin{cases}
	0  & \mbox{ if } k \leq M \\
	\sum_{j=1}^{k-1} e^{-i \omega j } \bar{c}_j \bar{c}_{k-j}  		& \mbox{ otherwise.}
	\end{cases}
	\label{eq:F_infty_center}
	\end{equation}
	As $ \bar{c}_j\bar{c}_{k-j}=0$ when either  $ j > M$ or $ k - j >M$, then it follows that:  
	\begin{align*}
|F_\infty ( \bar{\alpha}, \bar{\omega} , \bar{c}) |_k \leq \sum_{j=k-M}^{M} |\bar{c}_j \bar{c}_{k-j}|  .
	\end{align*}
	Noting that $ |F_\infty ( \bar{\alpha}, \bar{\omega} , \bar{c}) |_k=0 $ for $ k > 2 M$, we calculate: 
	\begin{align*}
		|F_\infty ( \bar{\alpha}, \bar{\omega} , \bar{c}) |_k &\leq  k^{-s}
		\max_{ M+1 \leq k_0 \leq 2M } k_0^{s} \sum_{j=k_0-M}^{M}   |\bar{c}_j \bar{c}_{k_0-j}| \\
		&= k^{-s}  g_\infty^i(X) .
	\end{align*} 	
\end{proof}

\begin{proof}[Proof of \eqref{eq:prop_giiB}]
	We show that $\left|h* c \right|_k   \leq \frac{1}{k^s}  g_\infty^{ii,a}(X) $ for $	M + 1 \leq  k  $ and all $ c \in \pi_c (X)$ and $h \in \pi_c (H)$. 
 	Fix $x=(\alpha , \omega , c) \in X$ and  $  h  \in \pi_c(H)$, and write $c_M = \pi_M(c), c_\infty= \pi _\infty(c), h_M = \pi_M(h)$, and $h_\infty = \pi_\infty(h)$.   
	We may expand $ h*c$ as follows: 
	\begin{equation}
		h*c = h_M * c_M + h_M * c_\infty + c_M * h_\infty + h_\infty*c_\infty.
		\label{eq:H*X}
	\end{equation}
	The composition $ h_M * c_M $ only has non-zero components for $ M+1 \leq k \leq 2M $, thereby it is bounded by the computable value below: 
	\begin{align}
		h_M *_k c_M &\leq \tfrac{1}{k^s}  \max \{ k_0^s \cdot h_M *_{k_0} c_M : M+1 \leq k_0 \leq 2 M \} \nonumber \\ 
&\leq		\frac{1}{k^s}   \max_{M+1 \leq k_0 \leq 2 M}  k_0^s  \sum_{j=k_0-m}^M |H|_j |X|_{k_0-j}   .
		 	\label{eq:HM*XM}
	\end{align}
	We calculate $c_M  * h_\infty$ for $ k \geq M+1$, noting that $ [h_\infty]_{k-j}=0$ if $ k-j \leq M$, as below:
	\begin{align*}
			c_M 	 *_k h_\infty  =& \;
			  \sum_{j=1}^{k-1} [c_M ]_{j} [h_\infty]_{k-j}  
			+ \sum_{j=1}^{\infty} [c_M^* ]_{j} [h_\infty]_{k+j}
			+ [c_M]_{k+j} [h_\infty^*]_{j} 
			\\
			=&	\sum_{j=k-M-1}^{M} [c_M]_{j} [h_\infty]_{k-j}  
			+ \sum_{j=1}^{M} [c_M^*]_{j} [h_\infty]_{k+j} 
	\end{align*}
	Using the estimates $ |c_j| \leq |X|_j$ for $j \leq M$ and $ |h_j| \leq C_0 / j^s$ for $ j \geq M+1$, we calculate the following: 	
	\begin{align}
	|c_M  * h_\infty |_k \leq&	
	\sum_{j=k-M-1}^M | X|_j  \frac{C_0}{(k-j)^s}
	+ \sum_{j=1}^{M} | X|_j \frac{C_0}{(k+j)^s} 
	\nonumber 
	\\
	\leq&  \frac{C_0}{k^s}  \left( 	
	\sum_{j=k-M-1}^{M} | X|_j \left( \frac{k}{k-j} \right)^s 
	+ \sum_{j=1}^M | X|_j   \right) .
	\label{eq:SumWeird}
	\end{align} 
	Note that $\tfrac{k}{k-j}$ is decreasing with $k$.
	To maximize the coefficient of $ | X|_j $ in the first sum of \eqref{eq:SumWeird}, we choose the smallest $k$ such that $ j \leq k - M -1$. 
	Hence, for each coefficient, we choose $k=M+j+1$ as an upper bound.  
 We obtain the following:
\begin{equation}
\label{eq:XM*Hinfty}
		|c_M * h_\infty|_k \leq \frac{C_0}{k^s}   \sum_{j=1}^M |X|_j  \left(  \left( \frac{M+j+1}{M+1} \right)^s +1 \right) .
\end{equation}
	An analogous calculation produces a  bound for $ | h_M * c_\infty|$ as given below:	
\begin{equation}
\label{eq:HM*Cinfty}
|h_M * c_\infty|_k \leq \frac{C_0}{k^s}   \sum_{j=1}^M |H|_j  \left(  \left( \frac{M+j+1}{M+1} \right)^s +1 \right) .
\end{equation}
	Lastly we estimate $|h_\infty * c_\infty|_k$.  
	For  $h_\infty,c_\infty \in H_\infty$ and $ k \geq M+1$ we calculate:
	\begin{align*}
	h_\infty * c_\infty 
	&= \sum_{j=1}^{k-1} [h_\infty]_j [c_\infty ]_{k-j} 
	+ 	\sum_{j=1}^\infty [h_\infty^*]_j  [c_\infty ]_{k+j} 
						+ [h_\infty]_{k+j} [c_\infty^{*}]_j \\
	&= \sum_{j=M+1}^{k-M-1} [h_\infty]_j [c_\infty ]_{k-j}  + 
		\sum_{j=M+1}^\infty[h_\infty^*]_j  [c_\infty ]_{k+j} 
		+ [h_\infty]_{k+j} [c_\infty^{*}]_j . 
	\end{align*}
Taking norms and using the estimate $ |h_j| \leq C_0/j^s$ for $M+1 \leq j$ we obtain:  
	\begin{align*}
	|h_\infty * c_\infty|_k 
	&\leq  \sum_{j=M+1}^{k-M-1} \frac{C_0^2}{j^s(k-j)^s} + 2  \sum_{j=M+1}^{\infty} \frac{C_0^2}{j^s(k+j)^s} \\
	&\leq C_0^2  \left( \sum_{j=M+1}^{k-M-1} \frac{1}{j^s(k-j)^s} \right)  +   \frac{2}{k^s} \frac{C_0^2}{(s-1)M^{s-1}}  .
	\end{align*}
The remaining sum is only nonzero for $ k \geq 2 (M+1)$, and we bound it as follows: 
	\begin{align*}
		\sum_{j=M+1}^{k-M-1} \frac{1}{j^s(k-j)^s}  
		&= \frac{1}{k^s} \sum_{j=M+1}^{k-M-1} \left( \frac{1}{j} + \frac{1}{k-j} \right)^s \\
		&\leq \frac{2}{k^s} \sum_{j=M+1}^{k/2} \left( \frac{2}{j}  \right)^s \\
		&\leq    \frac{2^{s+1}}{k^s(s-1)} \left(\frac{1}{M^{s-1}} - \frac{1}{(k/2)^{s-1}} \right).
	\end{align*}
	This estimate is maximized in the $ \| \cdot \|_s$ norm  by taking $k \to \infty$. 
	Thereby, we obtain the following estimate:
	\begin{equation}
	\label{eq:Hinfty*Hinfty}
			\left| h_\infty * c_\infty \right|_k \leq  \frac{1}{k^s} \frac{2 C_0^2 (2^s+1)}{(s-1)M^{s-1}}.
	\end{equation}
	By combining the results from (\ref{eq:HM*XM}~-~\ref{eq:Hinfty*Hinfty})  into \eqref{eq:H*X}, it follows that  if $	M + 1 \leq  k  $, then $\left|h* c \right|_k   \leq \frac{1}{k^s}  g_\infty^{ii,a}(X) $.
	
\end{proof}

\begin{proof}[Proof of \eqref{eq:prop_giiC}]
	We show that $ \left|(\K^{-1} \pi_c(x_1)) * \pi_c(x_2) \right|_k  \leq \frac{1}{k^{s-1}} g_\infty^{ii,b}(X) $ for $
	M + 1 \leq  k $ and all $ x_1,x_2 \in X$. 
	For $ i=1,2$ let us fix $c_i  \in \pi_c(X)$ and recall that $ C_1 \geq \| c_i \|_s$  by Definition \ref{def:Gfunctions}. 
	We can write $  (\K^{-1} c_1 )*_k c_2 $ as below:
	\begin{align}
		(\K^{-1} c_1 )*_k c_2
		&= 
		\sum_{j = 1 }^{k-1}  j [c_1]_{j} [c_2]_{k-j}  + 	
		\sum_{j = k+1 }^{\infty}  j [c_1^*]_{j} [c_2]_{k+j} + (k+j) [c_1]_{k+j} [c_2^*]_j .
	\end{align}
	Using $ |c|_j \leq C_1 / j^s$ and $ |c|_{k+j} \leq C_0 / (k+j)^s$ for $ k \geq M+1$, we obtain a bound on $ | (\K^{-1} c_1 )* c_2|_k$ as below: 
	\begin{align*}
	 | (\K^{-1} c_1 )* c_2|_k
	&\leq 
	\sum_{j = 1 }^{k-1} \frac{j C_1 C_1}{j^{s}(k-j)^s}  + 	
		\sum_{j = 1}^{\infty}  \frac{ C_1 C_0}{j^{s-1}(k+j)^s} +  \sum_{j = 1}^{\infty} \frac{C_0 C_1}{(k+j)^{s-1}j^{s}} \\	
	&\leq 
	C_1^2 \left(\sum_{j = 1 }^{k-1} \frac{1 }{j^{s-1}(k-j)^s} \right)
	+ 	
	\frac{C_1 C_0}{(k+1)^s}\left( 1 + \frac{1}{s-2}\right)  
	+ 	
	\frac{C_1 C_0}{(k+1)^{s-1}}\left( 1 + \frac{1}{s-1}\right).
	\end{align*}
	Since $ 5 \leq M$, thereby  $6 \leq M+1 \leq k$ and by 
 Lemma \ref{prop:DiscreteConv} we can simplify the remaining sum as follows: 
	\begin{align*}
	\sum_{j =1}^{k-1} \frac{1}{j^{s-1} (k-j)^s} = \frac{k}{2}
	\sum_{j =1}^{k-1} \frac{1}{j^{s} (k-j)^s} \leq \frac{k}{2} \frac{\gamma_k}{k^s} \leq    \frac{\gamma_{M+1}}{2 k^{s-1}}.
	\end{align*}
	Taking $ k \geq M+1$, it follows that:  
	\begin{align*}
		|(\K^{-1} \pi_c(x_1)) * \pi_c(x_2)|_k 
		&\leq 
		\frac{1}{k^{s-1}} \left( 
		\frac{C_1^2 \gamma_{M+1}}{2} 
		+ C_1C_0 \left( \frac{s-1}{(M+2)(s-2)} + \frac{s}{s-1} \right)  \right) \\
		&= \frac{1}{k^{s-1}} g_\infty^{ii,b}(X) .
	\end{align*}
\end{proof}


\section{Bounding the Krawczyk Operator}
\label{sec:Krawczyk}

When defining a Krawczyk operator $K(X,\bar{x})$ for a function $f: Y\to Z$ one must choose a linear operator $ A^\dagger : Z \to Y$. 
The map $A^\dagger$ is  typically chosen to approximate $Df(\bar{x})^{-1}$.  
 Even in finite dimensions it may be impossible to exactly calculate the inverse of a matrix using floating point arithmetic. 
To denote a fixed but numerically approximate definition, we introduce the notation $: \approx$.  
Since we set up our theorems in an \emph{a posteriori} format, the question of whether our numerical approximation is sufficiently accurate is answered by whether our computer-assisted proof is successful or not.

As with any method relying on a contraction mapping argument, the Krawczyk operator is only truly effective in locating the zeros of a function if they are isolated.  
Since the non-trivial zeros of $F$ are not isolated, and in fact form a 2-manifold \cite{regala1989periodic}, we do not define a Krawczyk operator corresponding  directly to $F:\R ^ 2 \times \oo^s \to \oo^{s-1}$. 
We must first reduce the dimensionality of its domain by two. 

We  reduce one of the dimensions by imposing a phase condition; we may assume without loss of generality that the first Fourier coefficient is a positive real number (see Proposition \ref{prop:TimeTranslation}).  
To that end, we define a codimension$-1$ subspace $ \tilde{\oo}^s \subseteq \oo^s$ as follows: 
\[
 \tilde{\oo}^s := \{ c \in \oo^s : c_1=c_1^*\}.
\]
To reduce the other dimension, we consider $\alpha$ as a parameter and perform our estimates uniformly in $\alpha$.

For a cube $ X \subseteq \R ^ 2 \times \tilde{\oo}^s$  we define a Krawczyk operator to find the zeros of functions  $ F_\alpha :  \R^1 \times \tilde{ \oo}^{s} \to \oo^{s-1}$ for all
 $\alpha \in \pi_{\alpha}(X). $
To that end, we would like to define a map $ A^\dagger $ to be an approximate inverse of the derivative $ DF_{\bar{\alpha}}(\bar{\omega},\bar{c}) \in \cL( \R ^1 \times \tilde{\oo}^s,\oo^{s-1})$ for some $ ( \bar{\alpha} , \bar{\omega} , \bar{c}) \in X$.  
We construct this approximate inverse by combining $A^\dagger_M$, a $2M \times 2M$ real matrix on the lower Fourier modes, with the operator $ - ( i\tfrac{\bar{\alpha}}{\bar{\omega}}) \cK \pi_\infty'$ on the higher Fourier  modes. 

As is ever the case, we may only explicitly perform a finite number of operations on  fundamentally finite dimensional objects, 
and because of this we defined Galerkin projections in \eqref{eq:Galerkin1} and \eqref{eq:Galerkin2}. 
To ensure the sum $F = F_M + F_\infty$ makes sense, the maps $ \pi_M, \pi_M'$ are defined to be but finite rank maps onto a subspace of an infinite dimensional Banach space. 
To emphasize this finite dimensional subspace as a space in its own right, as well as the new domain $\R ^1 \times \tilde{\oo}^s $, we define the following projection and inclusion maps: 
\begin{align*}
	\tilde{\pi}_M  	&: \oo^s \twoheadrightarrow \R ^ {2M} ,& 
	\tilde{\pi}_M' 	&:\R ^1 \times  \tilde{\oo}^s \twoheadrightarrow \R ^ {2M} 	,&  
	\tilde{i}_M		&:   \R ^ {2M} \hookrightarrow  \oo^s ,& 
	\tilde{i}_M'	&: \R ^ {2M} \hookrightarrow \R ^1 \times \tilde{ \oo}^s.
\end{align*}
\begin{align*}
\tilde{\pi}_M \circ \tilde{i}_M 	&= id_{\R ^{2M}}, & 
\tilde{\pi}_M' \circ \tilde{i}_M'	&= id_{\R ^{2M}} ,&  
\tilde{i}_M \circ \tilde{\pi}_M 	&= id_{\oo^s}  ,& 
\tilde{i}_M' \circ \tilde{\pi}_M'	&= id_{\R ^1 \times \tilde{\oo^s}}.
\end{align*}
We  define the linear operator $A^\dagger$ below in Definition \ref{def:Adagger} as follows:
We note that $A^\dagger$ will be injective if the $2M\times 2M$ matrix $A^{\dagger}_M$ has rank $2M$. 

\begin{definition} 
	\label{def:Adagger}
	Fix a cube $X \subseteq \R^2  \times \tilde{\oo}^{s}$.  
	For a point $(\bar{\alpha},\bar{\omega},\bar{c}) =  \bar{x}    \in X$ such that $ \bar{x} = \pi'_M(\bar{x} )$, define the following linear operators: 
	\begin{align*}
	A_M 
		&:\approx \, \tilde{\pi}_M \circ D     F_{\bar{\alpha}} (\bar{\omega},\bar{c}) \circ \tilde{i}_M' 
 		&A_M &\in \mathcal{L}(\R^{2M},\R^{2M})
 	\\
	A_M^\dagger 
		& :\approx A_M^{-1}
		&A_M^\dagger &\in \mathcal{L}(\R^{2M},\R^{2M}) 
	\\
	A(\bar{x},M) 
		&:= \tilde{i}_M \circ A_M \circ \tilde{\pi}'_M + i \tfrac{\bar{\omega}}{\bar{\alpha}} \K^{-1} \pi'_{\infty} 
		&A(\bar{x},M)  & \in  \mathcal{L}(\R^1 \times \tilde{\oo}^{s} , \oo^{s-1} ) 
	\\  
	A^{\dagger}(\bar{x},M)   
		&:= \tilde{i}'_M \circ A^{\dagger}_M \circ \tilde{\pi}_M - i \tfrac{\bar{\alpha}}{\bar{\omega}}   \K \pi_{\infty} 
		&A^{\dagger}(\bar{x},M)  & \in  \mathcal{L}( \oo^{s-1}, \R^1 \times \tilde{\oo}^{s}  ) .
 	\end{align*}
\end{definition}

While a Krawczyk operator $K(X,\bar{x})$ given as in Definition \ref{def:Krawczyk} is sufficient from a mathematical perspective, from a computational perspective it leaves something to be desired. 
We address this deficiency in Definition \ref{def:KrawczykApprox} by defining an explicitly computable operator $ K'(X,\bar{x})$ as  an outer approximation to $K(X,\bar{x})$, which is to say that  $K(X,\bar{x}) \subseteq K'(X,\bar{x})$. 
In Theorem \ref{prop:K_Inclusion} we prove this, and in Theorem \ref{prop:KrawczykOuterApprox} we give an analogue of Theorem \ref{prop:Krawczyk}.

In practice, use \emph{interval arithmetic}  \cite{moore2009introduction}  to compute an outer approximations for the arithmetic combination of sets (e.g. $A + B = \bigcup_{a\in A, b\in B} a + b$).
 This allows us to bound the image of functions over rectangular domains,  which is to say domains  given as the product of intervals. 
By employing outward rounding,  interval arithmetic can be rigorously implemented on a computer \cite{rump1999intlab}. 
 In every step an outer approximation is constructed as a rectangular domain, and the end result will  too be an outer approximation. 
While obtaining a tight approximation is desirable, it is not required; as long as we have an outer approximation, that is sufficient.

\begin{definition}
	\label{def:KrawczykApprox}

	Fix  a cube  $X\subseteq \R^2 \times \tilde{\oo}^s$ as in Definition \ref{def:cube} with $M \geq 5$, $s >2$ and $C_0 >0$. 
	Fix some $ \bar{x} = ( \bar{\alpha} , \bar{\omega},\bar{c}) \in X$ such that $\bar{x} = \pi'_M(\bar{x})$ and $ \Delta_{\omega} \geq \sup_{x \in X}  | \pi_\omega(x)-\bar{\omega}| $. 
	Fix $ A := A(\bar{x},M)$ and 
	$ A^{\dagger } := A^\dagger(\bar{x},M)$ as in Definition \ref{def:Adagger}.   
	Define the following functions:  
	\begin{align}
	g_\infty^{ii}(X) := & \frac{2 \bar{\alpha}}{\bar{\omega} (M+1)} g_\infty^{ii,a}(X)  +
	\sup_{\alpha \in \pi_\alpha(X)}
	\Delta_\omega
	\tfrac{  \bar{\alpha }}{\bar{\omega}}  \left(  ( \alpha^{-1} +1) C_0 + g_\infty^{ii,b}(X)  \right) \nonumber \\
	&\;+  \sup_{\alpha \in \pi_\alpha(X),\omega \in \pi_\omega(X)}
	\left( |1- \tfrac{\bar{\alpha}}{\alpha} \tfrac{\omega}{\bar{\omega}} | + \frac{\bar{\alpha}}{\bar{\omega} ( M+1)} \right) C_0 
	 \label{eq:giiInfty}  \\
	g_M(X)  := & g_M^{i}(X)  +  g_M^{ii}(X) \\
 	g_\infty(X)  := & \tfrac{\bar{\alpha}/\bar{\omega}}{M+1} g_\infty^{i}(X)  +  g_\infty^{ii}(X) .
	\end{align}
	Define $	K'(X,\bar{x}) :=  K'_M(X,\bar{x})  \times K'_\infty(X,\bar{x})$ by:  
	\begin{align}
	K'_M(X,\bar{x})  :=& \, \bar{x} - A_M^{\dagger} F_M( \bar{x}) +
	(I_M - A_M^{\dagger}A_M) \cdot \pi_M' (X-\bar{x}) \nonumber 
	\\
	& + 		A_M^{\dagger}( A_M- DF_M(X))(X-\bar{x}) 
		\pm A_M^\dagger  g_M(X) 
	\label{eq:K'M}
	\\
	K'_\infty(X,\bar{x}) :=&	\left\{ c_k  \in \C : |c_k| < g_\infty(X) /k^s  \right\}_{k=M+1}^\infty, 
	\end{align}	
	where $F_M(\bar{x}) \subseteq \R^{2M}$ is calculated to include the image of $F_M(\bar{x})$ for all $ \alpha \in \pi_\alpha (X)$, 
	where $DF_M(X) \subseteq \cL(\R^{2M},\R^{2M}) $ is  calculated to include the image of  ${\tilde{\pi}}_M \circ DF_\alpha(\omega,c) \circ \tilde{i}_M'$ for all $(\alpha,\omega,c) \in X$, 
 	and where $ \pm  A^\dagger_M g_M(X) \subseteq \R^{2M}$ is calculated to be a set satisfying: 
	\begin{equation*}
	\bigcup_{|v|_k \leq |g_M(X)|_k} A_M^\dagger \cdot v \subseteq \pm A^\dagger_M  g_M(X).
	\end{equation*}
\end{definition}

\begin{theorem}
	\label{prop:K_Inclusion}
	Fix a cube $X$ as in Definition \ref{def:cube} with $M \geq 5$, $s >2$ and $C_0>0$.  
	Fix a point $ \bar{x}   \in X$ such that $\bar{x} = \pi'_M(\bar{x} )$, and fix $ A := A(\bar{x},M)$, 
	$ A^{\dagger } := A^\dagger(\bar{x},M)$ as in Definition \ref{def:Adagger}.  
	Fix some $ \alpha \in \pi_\alpha(X)$, and  for  $ f \equiv F_\alpha :\R ^1 \times \tilde{\oo}^s \to \oo^{s-1}$ let $K$ be given as in Definition \ref{def:Krawczyk}. 
	Then   $ K(X,\bar{x} ) \subseteq K'(X,\bar{x}).$	
\end{theorem}

\begin{proof}
	Let $H := X - \bar{x}$.  
	We begin by proving that $\pi_M'  (K(X,\bar{x})) \subseteq  \pi_M' ( K'(X,\bar{x}))$, first showing that: 
	\begin{align}
	\pi_{M}' \circ	 (I- A^\dagger  DF (X)) \cdot H 
	\subseteq &
	K'_M(X,\bar{x}) -\left(\bar{x} - A_M^{\dagger} F_M( \bar{x}) \right). 
	\label{eq:prop:K_M_inclusion}
	\end{align}
	Fix some $ x \in X$ and $h = (h_\omega,h_c) \in H$. 
	We start by adding and subtracting $ A^\dagger A$, rewriting the LHS of \eqref{eq:prop:K_M_inclusion} as follows: 
	\begin{align*}
	\pi_{M}' 	 (I- A^\dagger  DF (x)) \cdot h
	=& 
	(I_M-A^\dagger_M A_M)\cdot 	\pi_{M}' (h)  
	+	\pi_{M}' 	 A^\dagger (A-  DF (x)) \cdot h \\
	=& 
	(I_M-A^\dagger_M A_M)\cdot 	\pi_{M}' (h)  \\
	&+		A_M^{\dagger}( A_M- DF_M(x)) \cdot \pi_{M}' (h) + A^\dagger_M  \pi_M  DF_\infty (x) \cdot \pi_{M}' (h).
	\end{align*}
	By \eqref{eq:prop_finite_defect_w} and \eqref{eq:prop_finite_defect_c} it follows that $ | \pi_M D F_\infty (x) \cdot h |_k \leq [g_M^i (X) + g_M^{ii}(X)]_k $.  
	Thereby, it follows that:
	 $ 	A_M^\dagger   \pi_M D F_\infty (x) \cdot h  \subseteq \pm |A_M^\dagger |   \cdot  g_M(X)$ for all $x \in X$ and $h\in H$.
	Hence from the definition of $ K'(X,\bar{x}) $ given in  \eqref{eq:K'M}, then \eqref{eq:prop:K_M_inclusion} follows. 
	From \eqref{eq:F_infty_center} we have that $ \pi_M F_\infty(\bar{x}) =0$, hence $ \pi_M'  ( \bar{x } - A^\dagger F(\bar{x})) = \bar{x} -A_M^\dagger F_M(\bar{x})$. 
It then follows that $ \pi_M \circ K(X,\bar{x}) \subseteq K'_M(X,\bar{x})$.		
	\newline
	
We now prove that $\pi_\infty'(K(X,\bar{x})) \subseteq \pi_\infty'(K'(X,\bar{x}))$, first showing that: 
	\begin{align}
	\left\| \pi_{\infty}' \circ 	 (I- A^\dagger  DF (X)) \cdot (X - \bar{x})  \right\|_s   \leq& g_\infty^{ii}(X). 
	\label{eq:prop:F_infty_tail}
	\end{align}	
	Fix some $ x = (\alpha ,\omega,c) \in X$ and $h = (h_\omega,h_c) \in H$. 
	We start by adding and subtracting $ A^\dagger A$, rewriting the LHS of \eqref{eq:prop:F_infty_tail} as follows: 
	\begin{align*}
 	\pi'_{\infty} 	 (I- A^\dagger  DF (x)) \cdot h     
	 =&  \;  \pi'_{\infty} 	 (I- A^\dagger A )\cdot h  +   \pi'_{\infty} A^\dagger  (A- DF (x)) \cdot h  \\
	=&\;  \pi'_\infty \circ A^\dagger ( A - DF(x)) \cdot h \\
	=& \; 
	 \pi'_\infty \circ A^\dagger  \left(A- \tfrac{\partial}{\partial c} DF  (x) \right) \cdot h_c - 	 
	 \pi'_\infty \circ A^\dagger  \tfrac{\partial}{\partial \omega} DF  (x) \cdot h_\omega .
	\end{align*}
We calculate $- \pi_\infty A^\dagger \frac{\partial}{\partial \omega} F(x) \cdot h_\omega$ writing $\frac{\partial}{\partial \omega} F(x)$ as in \eqref{eq:dFdW} below:  
\begin{align*}
- 	\pi_\infty \circ A^\dagger  \tfrac{\partial}{\partial \omega} F (X)) \cdot h_\omega 
 =&
  -i \pi_{\infty}\frac{ \bar{\alpha} }{\bar{\omega}} \K \left( i \K^{-1} ( \alpha^{-1} I   -      U_{\omega}) c - i  ( \K^{-1} U_{\omega} c) * c \right) \cdot  h_\omega \\
=& h_\omega
 \frac{  \bar{\alpha }}{\bar{\omega}}  \pi_{\infty} \left(  ( \alpha^{-1} I   -      U_{\omega}) c -  \cK  (\K^{-1} U_{\omega} c) * c \right) .
\end{align*}	
Using $|c|_j \leq C_0 /j^s$ and \eqref{eq:prop_giiC} we obtain for $k \geq M+1$ that: 
\begin{align}
\left|	\pi_\infty \circ A^\dagger  \tfrac{\partial}{\partial \omega} F(x)) \cdot \Delta_\omega  \right|_k
\leq&
\Delta_\omega
\frac{  \bar{\alpha }}{\bar{\omega}}  \left(  ( \alpha^{-1} +1) \frac{C_0}{k^s} + \frac{1}{k} \frac{g_\infty^{ii,b}(X)}{k^{s-1}} \right) 
\nonumber
\\
\left\|	\pi_\infty \circ A^\dagger  \tfrac{\partial}{\partial \omega} F(x)) \cdot \Delta_\omega  \right\|_{s} 
\leq& 
\Delta_\omega
\frac{  \bar{\alpha }}{\bar{\omega}}  \left(  ( \alpha^{-1} +1) C_0 + g_\infty^{ii,b}(X)  \right).
\label{eq:dKdW}
\end{align}	
For $(\alpha , \omega, c ) \in X$  we calculate $ \pi_\infty A^\dagger( A-  \frac{\partial}{\partial c} F) \cdot h_c$ below:
\begin{align*} 
\pi_\infty \circ A^\dagger  (A- \tfrac{\partial}{\partial c} F (x)) \cdot h_c 
&= -i 	\frac{\bar{\alpha} }{\bar{\omega}} \K  
\left( 
\left( i\tfrac{ \bar{\omega}}{\bar{\alpha} }\K^{-1} - 
(i \tfrac{\omega}{\alpha} \K^{-1}  + U_\omega)
\right) h_c  -  (U_\omega h_c) * c - (U_\omega c) * h_c  \right) \nonumber \\
&= \pi_\infty \left( (1- \tfrac{\bar{\alpha}}{\alpha} \tfrac{\omega}{\bar{\omega}} ) I + i \tfrac{\bar{\alpha}}{\bar{\omega}} \K U_\omega \right) h_c - \pi_\infty i  \tfrac{\bar{\alpha}}{\bar{\omega}} \K \left( (U_\omega c) *h_c + (U_\omega h_c ) *c \right).
\end{align*}
	Taking norms and using \eqref{eq:prop_giiB} we obtain: 
	\begin{align}
\left\| \pi_\infty \circ A^\dagger  (A- \tfrac{\partial}{\partial c} F (x)) \cdot h_c  \right\|_s \leq& \left( |1- \tfrac{\bar{\alpha}}{\alpha} \tfrac{\omega}{\bar{\omega}} | + \frac{\bar{\alpha}}{\bar{\omega} ( M+1)} \right) C_0 + \frac{2 \bar{\alpha}}{\bar{\omega} (M+1)} g_\infty^{ii,a}(X). \label{eq:dKdC}
	\end{align}
	By combining \eqref{eq:dKdW} and \eqref{eq:dKdC} and taking a supremum over $\alpha$ and $\omega$, we obtain the definition of $g^{ii}_\infty$ in \eqref{eq:giiInfty}, whereby \eqref{eq:prop:F_infty_tail} follows.

	To show that $\pi_\infty K(X,\bar{x}) \subseteq K_\infty'(X,\bar{x})$ note that from \eqref{eq:prop_F_infty_center} it follows that: 
	\begin{align*}
		 \|\pi_\infty(\bar{x} - A^\dagger F(\bar{x} ) ) \|_s 
		 =  \| - i \tfrac{\bar{\alpha}}{\bar{\omega}} \cK  \pi_\infty F(\bar{x}) \|_s
		 \leq \frac{\bar{\alpha}/\bar{\omega}}{M+1}g_\infty^i( X).
	\end{align*}
	Expanding out $ \pi_\infty K(X,\bar{x})$, it  follows that:  
\begin{align*}
	\| 	\pi_\infty K(X,\bar{x}) \|_s 
	\leq  & \|\pi_\infty(\bar{x} - A^\dagger F(\bar{x} ) ) \|_s  + 
	\left\| \pi_{\infty} 	 (I- A  DF (X)) \cdot (X - \bar{x})  \right\|_s   \\
	\leq & \frac{\bar{\alpha}/\bar{\omega}}{M+1}g_\infty^i( X) + g_\infty^{ii}(X) = g_\infty(X).
\end{align*}
Thus $ 	\pi_\infty K(X,\bar{x}) \subseteq K_\infty'(X,\bar{x})$. 
Thus, we have proved both that $\pi_M'(K'(X,\bar{x}))  \subseteq  \pi_M'(K(X,\bar{x}))$ and    $\pi_\infty'(K'(X,\bar{x})) \subseteq \pi_\infty'(K(X,\bar{x}))$.  
	Hence it follows that  $ K(X,\bar{x}) \subseteq K'(X,\bar{x})$.   
	
\end{proof}

\begin{theorem}
	\label{prop:KrawczykOuterApprox}
	Fix a cube $X$ as in Definition \ref{def:cube} with $M \geq 5$, $s >2$ and $C_0>0$. 
	Fix a point $ \bar{x} \in X$ such that $ \bar{x} = \pi'_M(\bar{x} )$. 
	Let $K(X,\bar{x})$ and $K'(X,\bar{x})$ be given as in Definition \ref{def:Krawczyk} and \ref{def:KrawczykApprox} respectively. 
	If $ K'(X,\bar{x}) \subseteq X$,  and moreover $g_\infty(X) < C_0$ and:
	\[
	\tilde{\pi}_M' \left(  K'_M(X,\bar{x}) + A^\dagger_M F_M(\bar{x})\right) 
	\subseteq int(\tilde{\pi}_M'(X)) ,
	\] 
then for all  $\alpha \in \pi_\alpha(X)$ there exists a unique point $ \hat{x}_\alpha = ( \alpha, \hat{\omega}_\alpha , \hat{c}_\alpha)   \in X$ such that $ F(\hat{x}_\alpha ) =0$. 

\end{theorem}
\begin{proof}

Fix $\alpha \in \pi_\alpha(X)$. 
By Theorem \ref{prop:Krawczyk}, in order to show that there exists a unique solution  to $F_\alpha = 0$, it suffices to show that there is some $ 0 \leq \lambda < 1$ for which:
\[
(I-A^\dagger DF(X)) (X-\bar{x}) \subseteq \lambda(X-\bar{x}) .
\]

We find a $\lambda_M$ which works for the $\pi_M'$-projection and a $\lambda_\infty$ which works for the $ \pi_\infty'$-projection. 
Since $K(X,\bar{x}) \subseteq K'(X,\bar{x})$ by Theorem \ref{prop:K_Inclusion} and  $\tilde{\pi}_M' \left(  K'_M(X,\bar{x}) + A^\dagger_M F_M(\bar{x})\right) \subseteq int(\tilde{\pi}_M'(X))$, it follows from the definition of $K(X,\bar{x})$ in \eqref{eq:KrawczykDef}  that:  
\begin{align}
\nonumber
\tilde{\pi}_M '\left( 	K(X,\bar{x}) + A^\dagger F(\bar{x}) \right)
	&\subseteq int(\tilde{\pi}_M'(X))  
	\\
	\tilde{\pi}_M'\left((I-A^\dagger DF(X)) (X-\bar{x}) \right)
	&\subseteq int \left(\tilde{\pi}_M'(X  - \bar{x})  \right)
	\label{eq:CompactContainment}
\end{align}
Since $\tilde{\pi}_M'\left((I-A^\dagger DF(X)) (X-\bar{x}) \right)$ is compactly contained inside of $   \tilde{\pi}_M'(X  - \bar{x})   \subseteq \R ^ {2 M} $, there is some positive distance separating the LHS of \eqref{eq:CompactContainment} away from the boundary of $\tilde{\pi}_M'(X  - \bar{x})$. 
It follows that  there must exist some  $0 \leq \lambda_M < 1$ such that $\tilde{\pi}_M' \left((I-A^\dagger DF(X)) (X-\bar{x}) \right) \subseteq \lambda_M\cdot  \tilde{\pi}_M'(X  - \bar{x}) $. 

Since $ K'_\infty(X,\bar{x}) \subseteq  \pi_\infty' X$ it follows that $g_\infty(X) \leq C_0$, and by our additional assumption this is in fact a strict inequality. 
If we define $\lambda_\infty := g^{ii}_\infty(X)/C_0<1$, then by  \eqref{eq:prop:F_infty_tail} it follows that:
\[
\pi_{\infty} 	 (I- A^\dagger  DF (X)) \cdot (X - \bar{x})  \leq \lambda_\infty \pi_\infty(X- \bar{x}).
\]
If we   define $ \lambda := \max \{ \lambda_M,\lambda_\infty\} < 1$ then it follows that: 
\[
(I- A^\dagger  DF (X)) \cdot (X - \bar{x})  \leq \lambda (X- \bar{x}).
\]
By Theorem \ref{prop:Krawczyk} there exists a unique point $ \hat{x}_\alpha = ( \alpha, \hat{\omega}_\alpha , \hat{c}_\alpha)   \in X$ such that $ F_\alpha(\hat{\omega}_\alpha ,\hat{c}_\alpha ) =0$. 
Moreover, this is true for all $\alpha \in   \pi_\alpha (X)$. 
\end{proof}




\section{Pruning  Operator}
\label{sec:Prune}

For a given cube, we want to know if it contains any solutions to $ F=0$. 
We try to determine this by combining several different tests into one \emph{pruning} operator described in Algorithm \ref{alg:Prune}. 
It is called a pruning operator because even if we cannot determine whether a cube contains a solution, we may still be able to reduce the size of the cube without losing any solutions.

We describe the tests performed in Algorithm \ref{alg:Prune}. 
Most simply, if we can prove that $ | F(X)|_k>0$ for some $ 1 \leq k \leq M$, then $F$ has no zeros in $X$. 
From Lemma \ref{prop:zeroneighborhood2}, we know that if a cube has a small $ \| \cdot \|_{\ell^1}$ norm then it cannot contain any nontrivial zeros. 
Furthermore, if a cube is contained in the  neighborhood of the Hopf bifurcation explicitly given by Lemma \ref{prop:BifNbd}, then the only solutions that can exist therein  are on the principal branch. 
If none of those situations apply, then we calculate the outer approximation of the Krawczyk operator given in Definition  \ref{def:KrawczykApprox}. 
If the hypothesis of Theorem \ref{prop:KrawczykOuterApprox} is satisfied, then there exists a unique solution. 
Alternatively, if $ X \cap K(X,\bar{x}) = \emptyset$, then there do not exist any solutions in $X$. 
If none of these other situations apply, then we replace $ X $ by $ X \cap K(X, \bar{x})$. 
Algorithm \ref{alg:Prune} arranges these steps in order of ease of computation. 
 
\begin{algorithm}[Prune]
	\label{alg:Prune}
	Take as input a cube $X$ with $ M\geq 5$ and $s >2$. 
	The output is a pair $\{flag,X'\} $ where  $ flag \in \Z$ and $X'\subseteq X$ is a cube. 	
	\begin{enumerate}
		\item Compute $\delta :=2 \sum_{k=1}^M |X|_k+ \frac{2 C_0 }{(s-1) M^{s-1}} $. 
		\item 
		If for all $ (\alpha,\omega,\cdot) \in X$ we have $ \alpha \in (0,2]$, $\omega \geq 1.1$, and  $\delta < g(\alpha , \omega)$ for $g$ defined in  \eqref{eq:ZeroNbd},  then return $  \{1,\emptyset\}$. 
		\item If for all $ (\alpha,\omega,\cdot) \in X$ we have   $| \alpha -\pp | \leq 0.00553$,  $ | \omega - \pp | \leq 0.0924$ and $\delta < 0.18$, then return $  \{2,X\}$. 
		\item If $ \inf_{x \in X} |F_M(x)|_k > h_k(X)$ for $h_k$ defined in \eqref{eq:F_tail}  and some $ 1 \leq k \leq M$, then return $ \{1,\emptyset\}$. 
		\item Fix some $ \bar{x} \in X$ such that $\bar{x} = \pi_{M}'(\bar{x})$ and
			$ \pi_{M}'(\bar{x})$  is approximately the center of $ \pi_{M}'(X)$. 
		 Construct $K'(X,\bar{x}) $ as in  Definition \ref{def:KrawczykApprox}.
		\item If  $  K'(X,\bar{x}) \subseteq X $, $g_\infty(X)<C_0$, and $\tilde{\pi}_M \left(  K'_M(X,\bar{x}) + A^\dagger_M F_M(\bar{x})\right) 		\subseteq int(\tilde{\pi}_M(X))$, 
		then return  $\{3,X\}$.  
		\item If  $ X \cap K'(X,\bar{x}) = \emptyset$, then return $\{1,\emptyset\}$. 
		\item Else  return $ \{0,X \cap K'(X,\bar{x})\}$.
	\end{enumerate}
\end{algorithm}

\begin{theorem}
	\label{prop:Prune}
	Let $ \{flag,X'\}$ denote the output of Algorithm \ref{alg:Prune} with input a cube $X$. 
	\begin{enumerate}[(i)]
		\item If $flag=1$, then $ F(x) \neq 0$ for all nontrivial $ x \in X$. 
		\item If $flag =2$, then the only solutions to $F=0$ in $X$ are  on the principal branch.
		\item If $flag =3$, then for all $\alpha \in \pi_\alpha(X) $ there is a unique $\hat{\omega}_\alpha \in \pi_{\omega}(X)$ and $ \hat{c}_\alpha \in \pi_c(X)$  such that $F(\alpha,\hat{\omega}_\alpha,\hat{c}_\alpha)=0$.
		\item If there are any points $ \hat{x} \in X$ for which $ F(\hat{x})=0$, then $ \hat{x} \in X'$. 
	\end{enumerate}
\end{theorem}
\begin{proof}
	To prove $(i)$ we must check the output from Steps 2, 4, and 7. 
	To prove $(ii)$ we must check Step 3. 
	To prove $(iii)$ we must check Step 6. 
	The proof of $(iv)$ follows from $(i)$, $(ii)$, $(iii)$, and Step 8.
	We organize the proof into the steps of the algorithm. 
	\begin{enumerate}
		\item It follows from \eqref{eq:SumIntegral} that   $ \|c\|_{\ell^1} < \delta $ for all $ c  \in \pi_c(X)$. 
		\item Since  $\alpha \in (0,2]$ and $\omega \geq 1.1$,  Lemma \ref{prop:zeroneighborhood2}  applies. If  $ \|c\|_{\ell^1} < \delta  < g(\alpha , \omega)$, then by Lemma \ref{prop:zeroneighborhood2} the only solutions to $F(\alpha,\omega,c)=0$ are trivial, which is to say $ c =0$. 
		\item If Step 3 returns $flag=2$, then by Lemma \ref{prop:BifNbd} 	there is at most one SOPS $ c \in X$ with frequency $ \omega$, and it lies on the branch of SOPS originating from the Hopf bifurcation at $\alpha = \pp$. 
		\item  Suppose that $ \inf_{x \in X} |F_M(x)|_k > h_k(X)$ for some $ 1 \leq k \leq M$. 
		Since  $ \sup_{x \in X} |F_\infty(x)|_k < h_k(X)$ by  \eqref{eq:prop_F_tail}, it follows from the triangle inequality that for all $ x \in X$  we have: 
		\[
		|F(x)|_k \geq \inf_{x \in X} |F_M(x)|_k - \sup_{x \in X} |F_\infty(x)|_k >0.
		\]
		Hence $|F(x)|_k >0$, and so $X$ cannot contain any zeros of $F$.
		\item Note that $ K(X,\bar{x}) \subseteq K'(X,\bar{x})$ by Theorem \ref{prop:KrawczykOuterApprox}. 
		\item 
		If Step 6 returns $flag=3$, then the hypothesis of Theorem \ref{prop:KrawczykOuterApprox} is satisfied. 
		Hence for all $ \alpha \in \pi_\alpha(X) $  there is  a unique $\hat{\omega}_\alpha \in \pi_{\omega}(X)$ and $ \hat{c}_\alpha \in \pi_c(X)$  such that $F(\alpha,\hat{\omega}_\alpha,\hat{c}_\alpha)=0$. 	 
		\item 
		By Theorem \ref{prop:Krawczyk} all solutions in $X$ are contained in $K(X,\bar{x})$. 
		Hence, all of the zeros of $F$ in $X$ are contained in $ X \cap K(X,\bar{x}) \subseteq X \cap K'(X,\bar{x})$. 
		
		If  $ X \cap K'(X,\bar{x}) = \emptyset$ then  $ X \cap K(X,\bar{x}) = \emptyset$,	whereby there cannot be any solutions in $X$. 
		\item As proved in  Step 7, all solutions in $X$ are contained in $X \cap K'(X,\bar{x})$. 
	\end{enumerate}
\end{proof}



	\section{Global Bounds on the Fourier Coefficients}
	\label{sec:SolutionSpace}
	
The goal of this section is to construct a   bounded region in $\R^2 \times \oo^s$ which contains all of the nontrivial zeros of $F$.  
This is ultimately achieved  in Algorithm \ref{alg:Comprehensive}, which is discussed in Section \ref{sec:FourierProjWright}, along with other estimates pertaining specifically to Wright's equation.

In Section \ref{sec:FourierProj}, we discuss generic algorithms used to construct   bounds in Fourier space.  
Algorithm \ref{alg:FourierProjection} converts pointwise bounds on a periodic function and its derivatives into a cube containing  its Fourier coefficients. 
Algorithm \ref{alg:TimeTranslate} modifies a cube so that after a time translation, any periodic function contained therein will satisfy the phase condition $ c_1 = c_1^*$. 

	\subsection{Converting Pointwise Bounds into Fourier Bounds}
	\label{sec:FourierProj}

	To translate pointwise bounds on a periodic function into bounds on its  Fourier coefficients we use the unnormalized $L^2$ inner product,  which we define for $g,h \in L^2([0,2 \pi / \omega],\C)$ as: 
	\begin{equation}
	\left< g ,h \right> :=  \int_0^{2 \pi / \omega } g(t) h(t)^* \,dt.
	\label{eq:L2InnerProduct}
	\end{equation}
	For a function $y$  given as in  \eqref{eq:FourierEquation},  its Fourier coefficients may be calculated as $c_k =  \tfrac{1}{2 \pi / \omega} \left< y(t), e^{i \omega k t}  \right> $.  
	By applying \eqref{eq:L2InnerProduct} to \emph{a priori} estimates on $y$ we are able to derive bounds on its Fourier coefficients. 
	For example, in  \cite{wright1955non} it is shown that $ -1 < y(t) < e^\alpha -1$ for any global solution  to \eqref{eq:Wright}. 
	Hence, when $e^\alpha \geq 2$ the Fourier coefficients of any periodic solution to \eqref{eq:Wright} must satisfy $|c_k|  \leq \tfrac{1 }{2 \pi /\omega } ( e^\alpha -1)$ for all $ k \in \Z$.

	With more detailed estimates on $y$ we can produce tighter bounds on its Fourier coefficients. 
	In \cite{jlm2016Floquet,neumaier2014global} such estimates are numerically derived in a rigorous fashion. 
	One of the results from this analysis is a pair of bounding functions which provide upper and lower bounds on SOPS to \eqref{eq:Wright} at a given parameter value. 
	Formally, a \emph{bounding function} is defined to be an interval valued function $\chi(t) = [ \ell(t), u(t)]$ where $ \ell,u:\R\to  \R$.

 	These functions $\ell,u$ are constructed in \cite{neumaier2014global,jlm2016Floquet} using rigorous numerics, and in particular interval arithmetic. 
 	As a matter of computational convenience,  these functions are defined as piecewise constant functions which change value only finitely many times (see Figure \ref{fig:FourierDerivativeProjections}). 
For functions of this form, calculating a supremum over a bounded domain is reduced to finding the maximum of a finite set, and calculating an integral is reduced into a finite sum. 
For elementary functions   such as $\sin$ or $\cos$,  interval arithmetic packages have been developed which allow us to rigorously bound their image over arbitrary domains\cite{rump1999intlab}.

\begin{minipage}{\textwidth}	
	\begin{minipage}[b]{0.5\textwidth}
				\centering	
		\captionof{subfigure}{Bounds for $y$}
	    	\includegraphics[width = .8\textwidth]{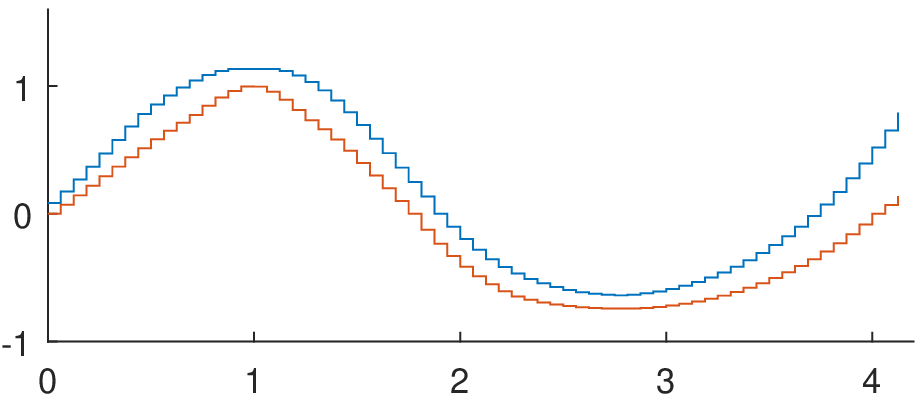} 	
		\[
		\begin{array}{c|c|c}
		k & A_{k,0} & B_{k,0} \\ 
		\hline
		1&	[   -0.103,    \;\;\;0.181] &[   -0.544,   -0.317] \\
		2&	[   -0.238,    \;\;\;0.110] &[   -0.142,   \;\;\; 0.187] \\
		3&	[   -0.207,    \;\;\;0.228] &[   -0.205,    \;\;\; 0.211] 
		\end{array}
		\]			
	\end{minipage}
	\hfill
	\begin{minipage}[b]{0.5\textwidth}
		\centering
				\captionof{subfigure}{Bounds for $y'$}
			    	\includegraphics[width = .8\textwidth]{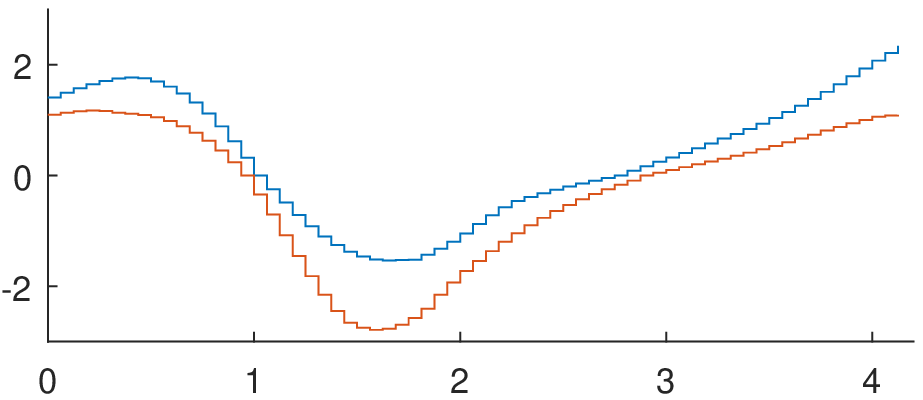} 	
		\[
		\begin{array}{c|c|c}
		k & A_{k,1} & B_{k,1} \\ 
		\hline
		1&	[   -0.154,    \;\;\;0.205] & [   -0.673,   		 -0.192] \\
		2&	[   -0.215,   	\;\;\;0.031] & [   -0.100,    \;\;\;0.179] \\
		3&	[   -0.094,    \;\;\;0.109] & [   -0.090,    \;\;\;0.125] 
		\end{array}
		\]	
	\end{minipage}	
	\\
	\begin{minipage}[b]{0.5\textwidth}
		\centering
		\vspace{4mm}
				\captionof{subfigure}{Bounds for $y''$}
		\includegraphics[width = .8\textwidth]{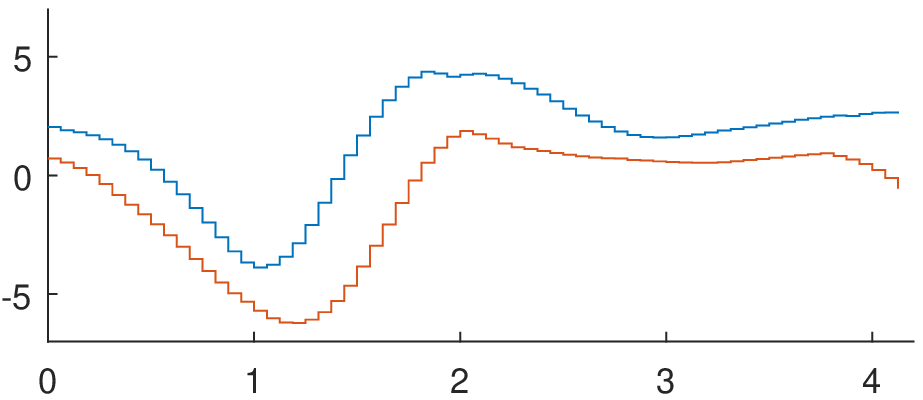} 		 				
		\[
		\begin{array}{c|c|c}
		k & A_{k,2} & B_{k,2} \\ 
		\hline
		1&		[   -0.384,    \;\;\; 0.525] &[   -0.848,    	 -0.103] \\
		2&		[   -0.205,    \;\;\; 0.037] &[   -0.094,    \;\;\;0.155] \\
		3&		[   -0.051,    \;\;\; 0.077] &[   -0.054,    \;\;\;0.071] 
		\end{array}
		\]	
	\end{minipage}
	\hfill
	\begin{minipage}[b]{0.49\textwidth}
		\centering	
				\vspace{4mm}
		\captionof{subfigure}{Bounds for $y'''$}
		\includegraphics[width = .8\textwidth]{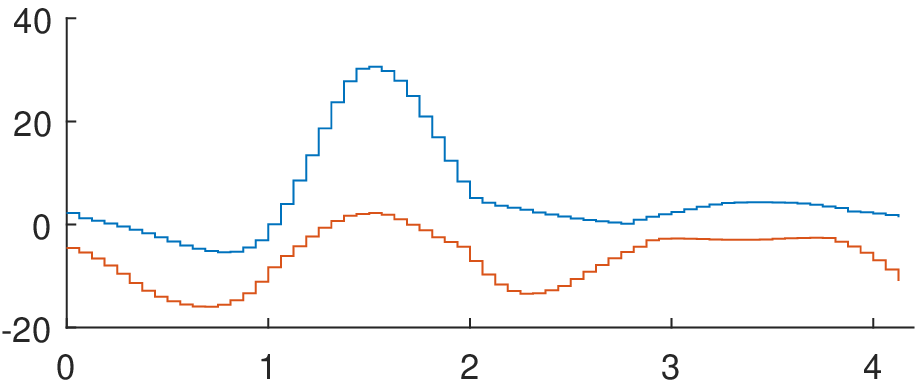} 		 				
		\[
		\begin{array}{c|c|c}
		k & A_{k,3} & B_{k,3} \\ 
		\hline
		1&		[   -0.995,    \;\;\; 1.160] &[   -1.713,    \;\;\;0.715] \\
		2&		[   -0.279,    \;\;\; 0.053] &[   -0.120,    \;\;\;0.194] \\
		3&		[   -0.039,    \;\;\; 0.068] &[   -0.045,    \;\;\;0.063] 
		\end{array}
		\]	
	\end{minipage}
\captionof{figure}{ \footnotesize
	Depicted in the  figures are functions $ \ell^{s},u^{s}: \R \to \R $ which bound a periodic function $y$ and its derivatives $y^{(s)}$.   
Depicted in the tables are the values for $ A_{k,s}$ and $B_{k,s}$ produced by Algorithm \ref{alg:FourierProjection} which  bound the Fourier coefficients $c_k = a_k + i b_k$ of $y$.  
}
\label{fig:FourierDerivativeProjections}
\end{minipage}
\vspace{4mm} 

Algorithm \ref{alg:FourierProjection} describes a method for obtaining rigorous bounds on the Fourier coefficients of a  periodic function $y$.
This algorithm applies the inner product $\left< \cdot , \cdot \right>$ to bounds  not just on the function $y$ but on its derivatives as well. 
Examples of these bounds are given in Figure  \ref{fig:FourierDerivativeProjections}, where we note  that by the third Fourier coefficient, the tightest estimate is given by the third derivative. 
We will use $ y^{(s)}$ denotes the $s$\textsuperscript{th} derivative of a function $y$, whereas we will use $Y^s$ to  denote a bounding function of index $s$, which bounds the derivative  $y^{(s)}$.

We have stated Algorithm \ref{alg:FourierProjection} so that it does not estimate the zeroth Fourier coefficient, as periodic solutions to \eqref{eq:Wright} necessarily have a trivial zeroth Fourier  coefficient. 
The algorithm could be modified in the obvious way to bound the zeroth Fourier coefficient of a function as well.

	\begin{algorithm}
		\label{alg:FourierProjection}  
		Take as input projection dimension $M \in \N$, period bounds $ [\underline{L},\overline{L}]$, and a collection of interval-valued functions:
		$$
		\left\{
		 Y^{s}(t) = [ \ell^{s}(t),u^{s}(t)] : \ell^{s},u^{s}:\R \to \R \right\}_{s=0}^S .
		$$	
		The output is an ($\alpha$-parameterless) cube $ X \subseteq \R^1 \times \oo^S$.  
		\begin{enumerate}
			\item  
			Define  
			$I_{\omega} := [2 \pi /  \overline{L} , 2 \pi/ \underline{L}]$. 
			\item  
			For $ 1 \leq k \leq M$ and $ 0 \leq s \leq S$ define $ \delta_c,\delta_s \in \R_+$ so that:   
			\begin{align*}
			\delta_{c} &\geq \sup_{ \omega \in I_\omega, y^{s} \in Y^{s}}  
			\int_{\underline{L}}^{\overline{L}}
			\left| \cos( \omega k t )  y^{s}(t)  \right| dt,  & 
			\delta_{s} &\geq
			\sup_{ \omega \in I_\omega, y^{s} \in Y^{s}}  
			\int_{\underline{L}}^{\overline{L}}
			\left| \sin( \omega k t )  y^{s}(t)  \right| dt,
			\end{align*}
			and define $ a^+_{k,s},a^-_{k,s},b^+_{k,s},b^-_{k,s} \in \R_+$ so that: 
			\begin{align*}
				a^+_{k,s} &\geq \;\;\; \delta_c +
				 \sup_{ \omega \in I_\omega, y^{s} \in Y^{s}}  
				 \int_{0}^{\underline{L}} \cos( \omega k t ) y^{s}(t)  dt
				 \\
				 a^-_{k,s} &\leq - \delta_c +
				  \inf_{ \omega \in I_\omega, y^{s} \in Y^{s}}  
				 \int_{0}^{\underline{L}} \cos( \omega k t ) y^{s}(t)  dt
				 \\
				 b^+_{k,s} &\geq \;\;\; \delta_s +
				 \sup_{ \omega \in I_\omega, y^{s} \in Y^{s}}  
				 \int_{0}^{\underline{L}} \sin( \omega k t ) y^{s}(t)  dt
				 \\
				  b^-_{k,s} &\leq -\delta_s +
				  \inf_{ \omega \in I_\omega, y^{s} \in Y^{s}}  
				 \int_{0}^{\underline{L}} \sin( \omega k t ) y^{s}(t)  dt. 
			\end{align*}
			
			\item  For $ 1 \leq k \leq M$ and $ 0 \leq s \leq S$ define: 
			\begin{align}
				A'_{k,s} &:= 
\frac{1}{2 \pi k^s}
				\left[ 
					 \inf_{ \omega \in I_\omega} \frac{a^-_{k,s}}{\omega^{s-1}},					
					\sup_{ \omega \in I_\omega} \frac{a^+_{k,s}}{\omega^{s-1}}
				\right], & 
								B'_{k,s} &:= 
				\frac{1}{2 \pi k^s}
				\left[ 
				\inf_{ \omega \in I_\omega} \frac{b^-_{k,s}}{\omega^{s-1}},					
				\sup_{ \omega \in I_\omega} \frac{b^+_{k,s}}{\omega^{s-1}}
				\right].
							\label{eq:A'B'}
			\end{align}
			Define the intervals $A_{k,s}$ and $B_{k,s}$ as follows:
			\begin{align*}
			A_{k,s} &:=
			\begin{cases}
			\;\;\;A'_{k,s}  & \mbox{ if }   s\equiv 0  \pmod 4  \\
			-B'_{k,s}  & \mbox{ if }   s\equiv 1  \pmod 4  \\				
			-A'_{k,s}  & \mbox{ if }   s\equiv 2  \pmod 4  \\
			\;\;\;B'_{k,s} & \mbox{ if }   s\equiv 3  \pmod 4 
			\end{cases} ,
			&
			B_{k,s} &:=
			\begin{cases}
			-B'_{k,s}  & \mbox{ if }   s\equiv 0  \pmod 4  \\
			-A'_{k,s}  & \mbox{ if }   s\equiv 1  \pmod 4  \\				
			\;\;\;B'_{k,s}  & \mbox{ if }   s\equiv 2 \pmod 4  \\
			\;\;\;A'_{k,s}  & \mbox{ if }   s\equiv 3  \pmod 4 
			\end{cases}.
			\end{align*}
			\item For $ 1 \leq k \leq M$ define: 
			\begin{align*}
			A_k &:= \bigcap_{0\leq s \leq S} A_{k,s} ,
			&
			B_k &:= \bigcap_{0\leq s \leq S} B_{k,s}.
			\end{align*}
			\item 
			For each $1 \leq k \leq M$, define $ \bar{a}_k := mid(A_{k,S})$, $ \bar{b}_k := mid(B_{k,S})$,  $ \bar{c}_k = \bar{a}_k + i \bar{b}_k$, and $ \bar{c}_{-k} = \bar{c}_{k}^*$.  
			Define $y_M^{S}(t,\omega) $ as in \eqref{eq:MidProjection}, and define $C_0>0$  so that \eqref{eq:TailBound} holds.   
			\begin{align}
			y_M^{S}(t,\omega) &:= \sum_{k=-M}^{M} \bar{c}_{k} (i \omega k )^S e^{i \omega k t} \label{eq:MidProjection}\\ 
			C_0 &\geq  \sup_{\omega \in I_{\omega }, y^{S} \in Y^{S} } \frac{1 }{2 \pi \omega^{S-1}} \int_0^{\overline{L}} \left| y^{S} (t)- y_M^{S}(t,\omega) \right|dt . \label{eq:TailBound}
			\end{align} 
			\item  
			Define a cube   $ X := X_M \times X_\infty \subseteq \R^1 \times \Omega^S$ by: 
			\begin{align*}
			X_{M} 		&:=   I_\omega  \times \prod_{k=1}^M A_k \times B_k  \\	
			X_\infty 	&:= \left\{ c_k  \in \C : |c_k| \leq C_0 /k^S  \right\}_{k=M+1}^\infty .
			\end{align*}		
			
		\end{enumerate}
	\end{algorithm}

	\begin{proposition}
		\label{prop:FourierProjection}
		Let the cube $X$   be the output of Algorithm \ref{alg:FourierProjection} with input $M \in \N$, $ [ \underline{L},\overline{L}] \subseteq \R$ and bounding functions $\{Y^{s}\}_{s=0}^{S}$. 
		Fix a function $ \hat{y}$   with period $L$ and continuous derivatives $ \hat{y}^{(s)}$ for $ 0 \leq s \leq S$.  
		If $L \in [\underline{L},\overline{L}]$ and  $ \hat{y}^{(s)}(t) \in Y^{s}(t)$ for all $ 0 \leq s \leq S$ and   $ t \in [0,\overline{L}]$, 
		then the  frequency and Fourier coefficients of $ \hat{y}$ satisfy $(\omega, \{c_k\}_{k=1}^\infty ) \in X$. 
	\end{proposition}
	\begin{proof}

		We organize the proof into the steps of the algorithm. 
		\begin{enumerate}
			\item  
			If the period of $\hat{y}$ is $L \in [ \underline{L} ,  \overline{L}]$ then  it will have frequency $ \hat{\omega} = 2 \pi / L$ and $ \hat{\omega} \in 
			[2 \pi /  \overline{L} , 2 \pi/ \underline{L}]$. 
			\item 
			
		Let us define 			
			\begin{align*}
			a_{k,s}   	&:= \left< \cos( \hat{\omega}  k t) , \hat{y}^{(s)}(t) \right>, &
			b_{k,s} 	&:= \left< \sin ( \hat{\omega}  k t) , \hat{y}^{(s)}(t) \right> .
			\end{align*}
			We show that $ a_{k,s} \in [a^-_{k,s},a^+_{k,s}]$. 	
			Since $L \in [ \underline{L} ,  \overline{L}]$  it follows that: 
			\begin{align}
			\left< \cos( \hat{\omega}  k t) , \hat{y}^{(s)} (t) \right> 
			 &=  	\int_0^{L } \cos( \hat{\omega}  k t ) \hat{y}^{(s)}(t)  dt \nonumber \\
			 &=  	\int_0^{\underline{L}} \cos( \hat{\omega}  k t ) \hat{y}^{(s)}(t)  dt 
			 +  	\int_{\underline{L}}^{L}\cos( \hat{\omega}  k t ) \hat{y}^{(s)}(t)  dt .
			 \label{eq:TailPeriod}
			\end{align}
			To estimate the rightmost summand in \eqref{eq:TailPeriod} we calculate:
			\begin{align*}
			\left| 	\int_{\underline{L}}^{L}\cos( \hat{\omega}  k t ) \hat{y}^{(s)}(t)  dt \right| \leq	\int_{\underline{L}}^{\overline{L}}  \left|  \cos( \hat{\omega}  k t ) \hat{y}^{(s)}(t)  \right| dt  
			\leq  \sup_{ \omega \in I_\omega, y^{s} \in Y^{s}}  
			\int_{\underline{L}}^{\overline{L}}
			\left| \cos( \omega k t )  y^{s}(t)  \right| dt \leq \delta_c.
			\end{align*}
			We obtain a bound on $a_{k,s}$ by appropriately taking an infimum and a  supremum in \eqref{eq:TailPeriod} as follows: 
			\begin{align*}
				\inf_{ \omega \in I_\omega, y^{s} \in Y^{s}}  
			\int_{0}^{\underline{L}} \cos( \omega k t ) y^{s}(t)  dt
			-\delta_c
				\leq 
				 a_{k,s}
				\leq   
				\sup_{ \omega \in I_\omega, y^{s} \in Y^{s}}  
			\int_{0}^{\underline{L}} \cos( \omega k t ) y^{s}(t)  dt
			+	\delta_c.
			\end{align*}
			Hence $ a_{k,s} \in [a^-_{k,s},a^+_{k,s}]$, and  by analogy $ b_{k,s} \in [b^-_{k,s},b^+_{k,s}]$.

			\item  
			Let $ c_k = a_k + i b_k$ denote the Fourier coefficients of $ \hat{y}$.  
			We show that $a_k \in  A_{k,s}$ and $b_k \in B_{k,s}$. 
			Firstly, we   calculate the  derivative $\hat{y}^{(s)}$ as follows:
			\[
			\hat{y}^{(s)}(t) = \sum_{k\in \Z} c_k (i \hat{\omega}  k)^s e^{i \hat{\omega}  k t}.
			\]
			We can express the Fourier coefficients of $ \hat{y}$ in terms of the Fourier coefficients of its derivatives $ \hat{y}^{(s)}$;  
			below, we calculate $c_k$ in terms of $ a_{k,s}$ and $ b_{k,s}$ as follows:  
			\begin{align}
			\int_0^{2 \pi / \hat{\omega}  }    c_k (i \hat{\omega}  k)^s e^{i \hat{\omega}  k t} \cdot e^{- i \hat{\omega} k t}dt    
			&=  			
			\left<  \hat{y}^{(s)}(t) ,e^{ i \hat{\omega}  k t} \right>
			\label{eq:FourierDerivativeInnerProduct}
			\\
			\frac{2 \pi }{\hat{\omega} } c_k ( i \hat{\omega}  k )^s			
			&= 		\left<  \hat{y}^{(s)}(t) ,\cos( \hat{\omega}  k t) \right> 
			-i \left< \hat{y}^{(s)}(t),\sin ( \hat{\omega}  k t)  \right> 
			\nonumber 			\\
			i^{s}a_{k}   +i^{s+1}b_{k}  
			&=   
			 \frac{a_{k,s} - i \, b_{k,s} }{2 \pi  \hat{\omega} ^{s-1}  k^{s}}  \nonumber 
			 .
			\end{align}
			From the definition of $ A'_{k,s}$ and $ B_{k,s}'$ in \eqref{eq:A'B'} it follows that:   
			\begin{align*}
				 \frac{a_{k,s} }{2 \pi  \hat{\omega}^{s-1}  k^{s}}  & \in A_{k,s}' ,&
 				 \frac{ \, b_{k,s} }{2 \pi  \hat{\omega}^{s-1}  k^{s}}  & \in B_{k,s}'. 
			\end{align*}
			By   matching the real and imaginary parts, which only depend on  $ s \pmod 4$, we obtain that $a_k \in A_{k,s}$ and $b_k \in B_{k,s}$.

			\item Since $a_k \in A_{k,s}$ and $b_k \in B_{k,s}$ for all $ k$ and $0 \leq s \leq  S$, it follows that: 
			\begin{align*}
				a_k &\in \bigcap_{0 \leq s \leq S} A_{k,s}, &
				b_k &\in \bigcap_{0 \leq s \leq S} B_{k,s} .
			\end{align*}
			\item 
			We calculate $c_k$ for $ k \geq M+1$ starting from \eqref{eq:FourierDerivativeInnerProduct}  and using the fact that the functions $e^{i \hat{\omega} k t}$ are $L^2$--orthogonal:  
			\begin{align*}
			c_k  ( i \hat{\omega }k )^S 
			&= \frac{1}{2 \pi / \hat{\omega}}
			\left< e^{i \hat{\omega } k t} , \hat{y}^{(S)}(t)  \right> \\
			&= \frac{1}{2 \pi / \hat{\omega}}
			\left< e^{i \hat{\omega } k t} , \hat{y}^{(S)}(t) 
			-  \sum_{j=-M}^{M} \bar{c}_{j} (i \hat{\omega } j )^S e^{i \hat{\omega } j t} \right>  \\
			&= \frac{1}{2 \pi / \hat{\omega}}
			\left< e^{i \hat{\omega } k t} , \hat{y}^{(S)}(t) -  y_M^S(t,\hat{\omega }) \right> .
			\end{align*}
			By taking absolute values, and the suprema over $ \omega \in I_\omega$ and $  y ^S \in Y^S$ we obtain the following. 
			\begin{align*}
			\left|   c_k  ( i \hat{\omega } k )^S \right|
			&\leq 
			 \frac{1 }{2 \pi / \hat{\omega }} 
			 \int_0^{L} 
			 \left|e^{-i \hat{\omega } k t}\right| \left| \hat{y}^{(S)} (t)- y_M^{S}(t,\hat{\omega }) \right|dt \\
			 |c_k| k^S
				&\leq \sup_{\omega \in I_{\omega } , y^{S} \in Y^{S}} \frac{1 }{2 \pi \omega^{S-1}} \int_0^{\overline{L}} \left| y^{S} (t)- y_M^{S}(t,\omega) \right|dt \\
				&\leq C_0.
			\end{align*}
			Hence $ |c_k| \leq C_0 / k^S$ for all $ k \geq M+1$. 
			\item  In Step 1 we showed that $ \hat{\omega} \in I_\omega$.  
			In Steps 2-4 we showed that $ c_k \in [X]_k$ for $ 1 \leq k \leq M$, and in Step $5$ we showed that $ |c_k| \leq C_0 / k^S$  for $ k \geq M+1$.

		\end{enumerate}
	\end{proof}

	\begin{algorithm}
		\label{alg:TimeTranslate}
		Take as input an  ($\alpha$-parameterless) cube $X \subseteq \R^1 \times \oo^s$. The output is an  ($\alpha$-parameterless) cube $ X' \subseteq \R^1 \times \tilde{\oo}^s$. 
	\end{algorithm}
	\begin{enumerate}
			\item For $ [X]_1 = A_1 \times B_1$, 
			with $ A_1=  [ \underline{A}_1 , \overline{A}_1 ]$ 
and 
 $B_1 = [ \underline{B}_1 , 
				 \overline{B}_1 ]$,
			define an interval $\Theta \subseteq \R$ so that: 
	\begin{align*}
\Theta & \supseteq  
\begin{cases}
	\bigcup_{a_1 \in A_1,b_1 \in B_1} 
	\;\;\;\tan^{-1}  ( b_1/a_1)
	& \mbox{ if } \underline{A}_1 > 0 
	\\
	\bigcup_{a_1 \in A_1,b_1 \in B_1} 
	\;\;\;\tan^{-1}  ( b_1/a_1) 
	+ \pi 
	& \mbox{ if } \overline{A}_1 < 0 
	\\
	\bigcup_{a_1 \in A_1,b_1 \in B_1} 
	-\tan^{-1} ( a_1/b_1) 
	+\pp
	& \mbox{ if } \underline{B}_1 > 0 
	\\
	\bigcup_{a_1 \in A_1,b_1 \in B_1} 
	-\tan^{-1} ( a_1/b_1) 
	-\pp 
	& \mbox{ if } \overline{B}_1 < 0 
	\\
	[-\pi,\pi]   & \mbox{ otherwise.}  
\end{cases} 
\end{align*}

			\item Rotate every Fourier coefficient's phase by $ - \Theta k$. 
			That is, define: 
			\begin{align*}
A_1' &:= \left[ \inf_{a_1 \in A_1 , b_1 \in B_1} \sqrt{a_1^2 + b_1^2}, \sup_{a_1 \in A_1 , b_1 \in B_1} \sqrt{a_1^2 + b_1^2}\right], 		& 	 B_1' := [0,0],
			\end{align*}
			and for $2 \leq  k \leq M$  define intervals $ A_k', B_k' \subseteq \R$ such that: 
			\begin{align*}
			A_k' &\supseteq 
			\bigcup_{\theta \in \Theta, a_k \in A_k, b_k \in B_k} \;\;\; \cos ( \theta k)  a_k + \sin( \theta k) b_k  
			\\ 
			B_k' &\supseteq
			\bigcup_{\theta \in \Theta, a_k \in A_k, b_k \in B_k} -\sin (\theta k)  a_k + \cos( \theta k) b_k.
			\end{align*}
			\item  
			Define a cube   $ X' := X_M' \times X_\infty' \subseteq \R^1 \times \Omega^S$ by 
			\begin{align*}
			X_{M}' 		&:=  I_\omega \times \prod_{k=1}^M A_k' \times B_k'  \\	
			X_\infty' 	&:= \left\{ c_k  \in \C : |c_k| \leq C_0 /k^S  \right\}_{k=M+1}^\infty .
			\end{align*}		
	\end{enumerate}

	\begin{proposition}
		\label{prop:TimeTranslation}
For an input cube $X $, let $ X'$ denote the output of Algorithm \ref{alg:TimeTranslate}. 
Suppose that $y:\R \to \R$ is a periodic function given as in   \eqref{eq:FourierEquation} with frequency and Fourier coefficients satisfying  $(\omega, \{c_k\}_{k=1}^\infty) \in X$. Then there exists some $ \tau \in \R $ such that the Fourier coefficients $c'$ of $ y(t+\tau)$ satisfy $(\omega,\{c'_k\}_{k=1}^\infty) \in X'$. Furthermore $c_1'$ is a real non-negative number.  
	\end{proposition}
	\begin{proof}
		We organize the proof into the steps of the algorithm.  		
		\begin{enumerate}
			\item 
			Write the first Fourier coefficient of $y$ as $ c_1 = a_1  + i b_1 $. 
			We may write $c_1  = r e^{i \theta}$  where $r = \sqrt{a_1^2 + b_1^2}$ and if $c_1 \neq 0$, then  $\theta$ is unique up to an integer multiple of $ 2 \pi$. 
			By the rules for $\arctan$ we can calculate: 
			\begin{equation*}
			\theta = 
			\begin{cases}
			\;\;\;\tan^{-1} (b_1 / a_1)  			& \mbox{ if } a_1 > 0 \\
			\;\;\;\tan^{-1} (b_1 / a_1)  + \pi 		& \mbox{ if } a_1 < 0 \\
			-\tan^{-1} (a_1 / b_1)  +\pp			& \mbox{ if } b_1 > 0 \\
			-\tan^{-1} (a_1 /b_1)  -\pp				& \mbox{ if } b_1 < 0. 
			\end{cases}
			\end{equation*}
			Since  $a_1 \in A_1 $ and $b_1 \in B_1 $, it follows that $ \theta \in \Theta$. 
			\item 
			For any $ \tau$ we can calculate the Fourier series of $ y(t + \tau)$ as follows: 
			\begin{align*}
			y(t+\tau)  = \sum_{k \in \Z} c_{k} e^{ i \omega k (t + \tau)} 
			=  \sum_{k \in \Z} c_{k}e^{ i \omega k \tau}  e^{ i \omega k t}.
			\end{align*}
			If we choose $ \tau = - \theta / \omega$, then   $ c_1' = c_1 e^{i \omega \tau} = \sqrt{a_1^2+b_1^2}$ is a real, non-negative number  and moreover $ c_1' \in [X']_1$.  
			\item  
			The Fourier coefficients of $ y(t + \tau)$ are given by $ c'_k = e^{-ik\theta} c_k$, hence $(\omega, \{c'_k\}_{k=1}^\infty) \in X'$. 
		\end{enumerate}
	\end{proof}

	\subsection{Bounds for Wright's Equation}
		\label{sec:FourierProjWright}

	The culmination of this subsection is  Algorithm 5.7 which, for a given range of parameters, constructs a collection of cubes covering the solution space to $ F_\alpha = 0$. 
	This algorithm begins with pointwise bounds on SOPS to \eqref{eq:Wright}. 
	To obtain these pointwise bounds, we use the results from \cite{jlm2016Floquet}. 
	One of the results \cite{jlm2016Floquet} achieves is, for a given range of parameters $I_\alpha$, it produces a collection of bounding functions $\cX$, such that if there is a SOPS to the exponential version of Wright’s equation at parameter $ \alpha \in I_\alpha$, then it will be bounded by one of the bounding functions in $\cX$. 
	Recall that solutions to the  exponential version of Wright’s equation solve  \eqref{eq:MNF} where $f(x) = e^x -1$, and can be transformed into the quadratic version of Wright's equation   \eqref{eq:Wright} using the change of variable $y = e^x -1$.

	As this is a computational result, it requires the selection of several computational parameters which, while immaterial to the proof, are necessary for implementation. 
	We describe them here with a brief description of \cite[ Algorithm 5.1]{jlm2016Floquet}. 
	To begin, this algorithm starts off with \emph{a priori} estimates, some of which are iteratively constructed, and require a selection of parameters  $ i_0, j_0  \in \N$. 
	These are used to construct numerical bounding functions having time resolution $ n_{Time} \in \N$. 
	A pruning operator is defined on these bounding functions, and the spacing between the zeros of a SOPS, and the parameter $N_{Period} \in \N$ defines how many times this pruning operator is applied in this initial construction of the bounding functions. 
	Then a branch and prune algorithm is executed, with a stopping criterion defined by the parameters $
	\epsilon_1,\epsilon_2  \in \R$. 
	We formally state the results of this algorithm below:	
	\begin{theorem}[See Theorem 5.2 in \cite{jlm2016Floquet}] 
		\label{prop:APbounds}
		Fix some $ I_\alpha = [ \alpha_{min }, \alpha_{max}]$ such that  $ \alpha_{min} \geq  \pp$. 
		Suppose that $x:\R \to \R $ is periodic with period $L$, and is a SOPS to \eqref{eq:MNF} at parameter $ \alpha \in I_\alpha$ with $ f(x) = e^x -1$.  
		Furthermore, assume without loss of generality that $ x(0) = 0 $ and $ x'(0) > 0$.

		If $\cL$ and $\cX$ denote the output of \cite[ Algorithm 5.1]{jlm2016Floquet} ran with input $I_\alpha$, then there exists some  $[\underline{L}_i, \overline{L}_i]\in \cL$ and $\chi_i \in \cX$ for which $L \in [\underline{L}_i, \overline{L}_i]  $ and $ x(t) \in \chi_i(t)$ for all $t$.    
	\end{theorem}

In \cite{jlm2016Floquet} the authors applied this algorithm to prove there is a unique SOPS for $ \alpha \in [1.9,6.0]$. 
However, one of the shortcomings of this algorithm  is that it has difficulty discarding low amplitude solutions near the Hopf bifurcation at $ \alpha = \pp$.  
To remedy this, we modify the pruning operator in \cite{jlm2016Floquet } with the addition of the following Proposition \ref{prop:PruningMod}. This allows for a new way to potentially conclude that a given bounding function cannot contain any SOPS. 
\begin{proposition}
	\label{prop:PruningMod}
	If $y$ is a nontrivial periodic solution to \eqref{eq:Wright} at parameter $ \alpha \in ( 0 , 2]$  and frequency $\omega \geq 1.1$, then:  
	\begin{equation*}
	\sup |y(t)| >   -\tfrac{1}{2} + \tfrac{1}{2}\sqrt{1 + \tfrac{4 \sqrt{3} \omega}{\pi \alpha } g(\alpha, \omega) }  .
	\end{equation*}
\end{proposition}
\begin{proof}
	Define $ M := \sup |y(t)|$. 
	From \cite[Lemma 4.1]{BergJaquette}  we know that if $F(\alpha,\omega,c)=0$,  then:
	\begin{align*}
	\| c \|_{\ell^1} 
	\leq  \frac{\pi}{\omega \sqrt{3}} \| y'\|_{\infty} 
	\leq \frac{\pi}{\omega \sqrt{3}} \alpha M (1 + M).
	\end{align*}
	From  Lemma \ref{prop:zeroneighborhood2}, the only solutions satisfying $ \| c\|_{\ell^1} < g(\alpha , \omega)$ are trivial. 
	Hence $(\alpha, \omega,c)$ would only be a trivial solution at best if the following inequality is satisfied: 
	\begin{equation*}
\| c\|_{\ell^1} \leq 	\frac{\pi}{\omega \sqrt{3}} \alpha M(1+M) <    g(\alpha ,\omega).
	\end{equation*}
	Solving the  quadratic equation $M^2 +M -   \frac{\omega \sqrt{3}}{\pi \alpha } g(\alpha ,\omega) <0$  produces the desired inequality. 
\end{proof}

The higher derivatives of a function can be very useful in constructing bounds on its Fourier coefficients and their rate of decay. 
While the bounding functions constructed in \cite{jlm2016Floquet} are not even continuous, we can use them to construct bounding functions for the derivative of SOPS to Wright’s equation via a bootstrapping argument. 
Namely, by taking a derivative on both sides of  \eqref{eq:Wright} we obtain an equation for the second  derivative of solutions to  \eqref{eq:Wright}. 
In a similar manner, can obtain an expression for the third derivative of solutions to \eqref{eq:Wright}, both of which are presented below:
\begin{align*}
y''(t) &= -\alpha 
\left[  y'(t-1) \left[1 + y(t) \right] +
y(t-1) y'(t)\right]  \\
y'''(t) &= - \alpha  \left[ y''(t-1) [1+y(t) ] + 2 y'(t-1)y'(t) + y(t-1) y''(t) \right] .
\end{align*}
Note that we can always express the derivative $ y^{(s)}(t)$ in terms of  $ y^{(r)}(t) $ and $ y^{(r)}(t-1) $ where $ 0 \leq r \leq s-1$. 
That is, we can inductively define functions $ f^s : \R^{2s} \to \R$ such that for all $t$ we have: 
\begin{align}
	y^{(s)}(t) = f^s\left(  y(t),y(t-1), y'(t),y'(t-1),\dots,  y^{(s-1)}(t),y^{(s-1)}(t-1)  \right). 
	\label{eq:Bootstrap}
\end{align}
If we start with a bounding function for $y$, then by appropriately adding and multiplying the bounding functions for $y^{(r)}$, taking wider bounds whenever necessary, we can obtain bounding functions for any derivative of $y$ (see for example Figure \ref{fig:FourierDerivativeProjections}).

Algorithm \ref{alg:Comprehensive} proceeds by first constructing 
bounding functions for $y$ and its derivatives, and then applying  Algorithm \ref{alg:FourierProjection} to obtain a cube containing its Fourier coefficients. 
Then it applies Algorithm \ref{alg:TimeTranslate} to impose the phase condition that $ c_1 = c_1^*$. 
\begin{figure}[h]
	\centering
	\includegraphics[width = .224\textwidth]{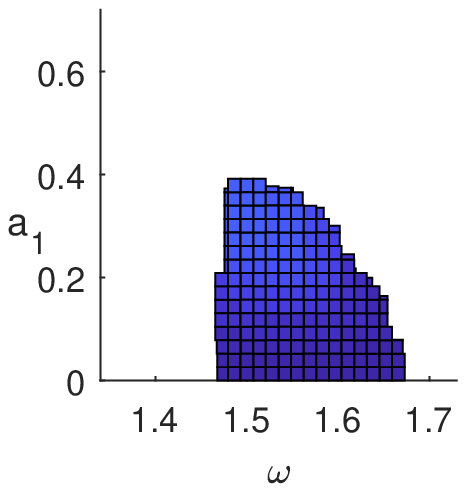}
	\includegraphics[width = .224\textwidth]{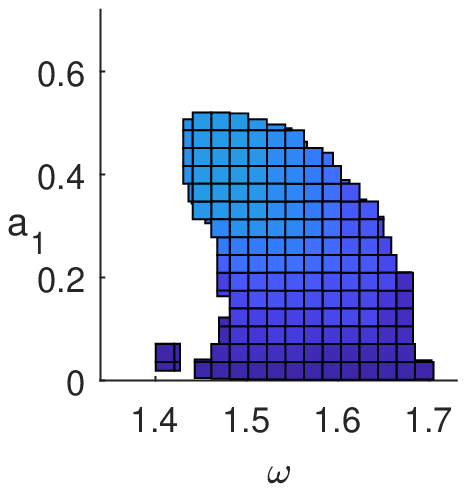} 
	\includegraphics[width = .224\textwidth]{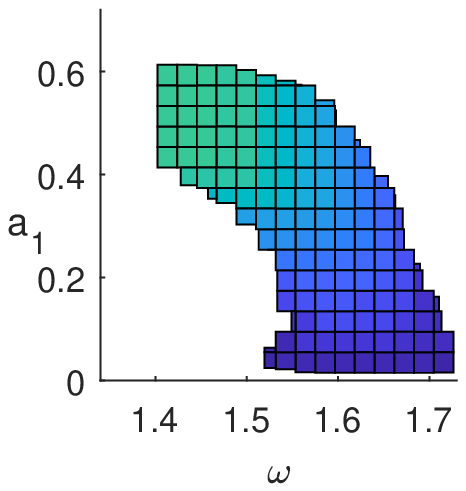}
	\includegraphics[width = .3025\textwidth]{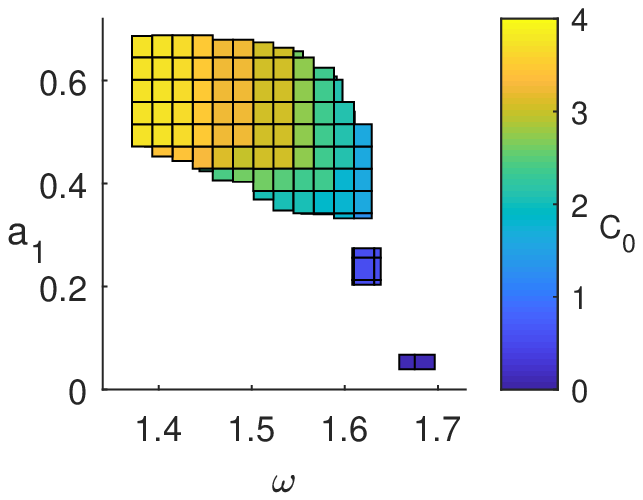}
	\caption{ \footnotesize
		Depicted above is the output of Algorithm \ref{alg:Comprehensive}  projected onto the  $(\omega,a_1)$ plane. 
		From left to right, the input $I_\alpha$ was taken to be $[\pp,1.6]$, $[1.6,1.7]$, $[1.7,1.8]$, and $[1.8,1.9]$. 
		Note that $C_0$ increases with $\alpha$, $a_1$, and period length $2 \pi / \omega$. }
	\label{fig:InitialPrep}
\end{figure}
In this manner  we obtain a collection of cubes which contains all of the Fourier coefficients to SOPS to \eqref{eq:Wright}.
We then apply Algorithm \ref{alg:Prune} to each cube,  discarding it if possible. 
This allows us to discard between 5\% and 60\% of cubes (see $N_{grid}'$ in Table \ref{table:RunTimes}). 

One problem however, is that the Fourier projection of two distinct bounding functions often overlap considerably. 
To address this we combine overlapping cubes together. 
While we could combine all of our cubes into one big cube, this would not be efficient. 
Instead, we  divide our cover along a grid in the $ \omega\times a_1$ plane (see Figure \ref{fig:InitialPrep}).

\begin{algorithm} 
	\label{alg:Comprehensive}
	Fix an interval of $ I_{\alpha} \subseteq [\alpha_{min},\alpha_{max}]$, integers  $M,S \in \N$  and a subdivision number $ N \in \N$, and the computational parameters for  \cite[Algorithm 5.1]{jlm2016Floquet}. The output is a (finite) collection of cubes $\cS = \{ X_i \subseteq \R ^2 \times \tilde{\oo}^s  \}$. 
	\begin{enumerate}
		\item Let $ \cX,\cL $ be the output of \cite[Algorithm 5.1]{jlm2016Floquet} with input $I_\alpha$ and appropriate computational parameters. 
		\item  
		Use the change of variables $ y = e^x -1$ to  define a collection of  functions:
		$$ \cY^{0} :=  \left\{ Y_i(t) = [ e^{\ell_i(t) } -1, e^{u_i(t) } -1]: \chi_i = [ \ell_i(t) , u_i(t)]  \in \cX \right\} .$$
		\item 
		Inductively define $ \cY^s$ for $ 1 \leq s \leq S$ so that corresponding to each $ Y_i^0 \in \cY^0$ there exists a 		$Y_i^s = [\underline{Y_i^s},\overline{Y_i^s}] \in \cY^s$   such that for $f^s$ defined in 		\eqref{eq:Bootstrap} we have: 
		\begin{align*}
		\underline{Y_i^s}(t) \leq& \inf_{\{y^{r}\}_{r=0}^{s-1} \in \{Y_i^r\}_{r=0}^{s-1}} 
		f^s\left(  y^0(t),y^0(t-1), \dots,  y^{s-1}(t),y^{s-1}(t-1) \right) \\
		\overline{Y_i^s}(t) \geq& \sup_{\{y^{r}\}_{r=0}^{s-1} \in \{Y_i^r\}_{r=0}^{s-1}} 
		f^s\left(  y^0(t),y^0(t-1), \dots,  y^{s-1}(t),y^{s-1}(t-1) \right).
		\end{align*}		
 
		\item Define $\cS' := \{ X_i' \subseteq \R^1 \times \oo^s   \}$ to be the collective output of Algorithm  \ref{alg:FourierProjection} run with $M \in \N$, and each of the sets $ L_i \in \cL$ and  $\{Y_i^{s}\}_{s=0}^S \in \{\cY^{s} \}_{s=0}^S $ as input.  
		\item  Define $\cS'' := \{ X_i'' \subseteq \R ^1 \times \tilde{\oo}^s\}$ to be the collective output of Algorithm \ref{alg:TimeTranslate} run with each of the sets $ X_i' \in \cS'$ as input. 
		\item 
		Define $\cS'''$ by taking the product of $ I_\alpha$ with the cubes in $ \cS''$. That is, define  $ \cS ''':= \{ I_\alpha \times X_i'' \subseteq \R ^2 \times \tilde{\oo}^s : X_i'' \in \cS''\}$.
		\item For  each $ X \in \cS'''$, let $ \{flag,X'\}$ denote the output of Algorithm \ref{alg:Prune} with input $X$. 
		If $flag=1$, then remove $X$ from $ \cS'''$. 
		Otherwise replace $X$ by $X'$.  
		\item  Subdivide the $ \omega \times a_1$ space covered by $ \cS'''$ into an $N \times N$ grid. 
		That is, define an index set $ B := \{ 1 ,2 , \dots, N\} \times \{1, 2, \dots,   N\}$ and define intervals $ I^\omega , I^{a_1} \subseteq \R$ so that: 
		\begin{align*}
		I^\omega 
		&\supseteq \bigcup_{X \in \cS'''} \pi_\omega(X),&
		 I^{a_1} 
		&\supseteq \bigcup_{X \in \cS'''} \pi_{a_1}(X).
		\end{align*}
		Subdivide $I^\omega$ and $I^{a_1}$ into $N$  subintervals of equal width,  $\{I^{\omega}_i\}_{i=1}^N$  and $\{I^{a_1}_i\}_{i=1}^N$, so that 
		$I^\omega  =  \bigcup_{i=1}^N I^\omega_{i} $ and $I^{a_1}  =  \bigcup_{i=1}^N I^{a_1}_{i} $.  
		\item 
		For each $ \beta =(\beta_1 , \beta_2) \in B$, take the union of cubes in $ \cS'''$ whose $ (\omega ,a_1)$--projection intersects $ I^\omega_{\beta_1} \times I^{a_1}_{\beta_2}$. That is, define: 
		\begin{align*}
		\tilde{X}_\beta &:= \{ (\alpha ,  \omega , c) \in \R^2 \times  \tilde{\oo}^s : \omega \in  I^\omega_{\beta_1}, [c]_1 \in I^{a_1}_{\beta_2}  \} ,
		\end{align*}
		and define $ X_\beta$ to be a cube such that: 
		\begin{align*}
			X_\beta &\supseteq \bigcup_{X \in \cS'''} X \cap \tilde{X}_\beta.
		\end{align*}
		\item  
		Define $ \cS := \{ X_{\beta} : \beta \in B\}$.
	\end{enumerate}
\end{algorithm}

\begin{theorem}
	\label{prop:Comprehensive}
	Fix an interval  $ I_{\alpha} = [\alpha_{min},\alpha_{max}]$ such that $ \alpha_{min} \geq \pp$, and let $\cS$ denote the output of Algorithm \ref{alg:Comprehensive}. 
	If a function $ y$ as given in \eqref{eq:FourierEquation} is a SOPS to Wright's equation at $ \alpha \in I_\alpha$, then there exists a time translation so that its Fourier coefficients are in $\bigcup \cS$. 
\end{theorem}

\begin{proof}
	Every SOPS $y$ to the quadratic version of Wright's equation given in \eqref{eq:Wright}  corresponds to a SOPS $x$ to the exponential version of Wright's equation  given in \eqref{eq:MNF} with $ f(x) = e^x -1$.  
	Fix  a SOPS $x : \R \to \R$  to the exponential version of Wright's equation with period $L$.   
We organize the proof into the steps of the algorithm. 
	\begin{enumerate}
		\item By Theorem \ref{prop:APbounds} there exists an interval $ L_i \in \cL$ and a bounding function $\chi_i \in \cX$ and  such that $L \in L_i$ and  $ x(t) \in \chi_i(t)$ for all $ t \in \R $. 
		\item The change of variables between the exponential and quadratic versions of Wright's equation is given by $y = e^x -1$. 	
		Hence for the interval $ L_i \in \cL$ and the bounding function $Y_i \in \cY^{0}$, it follows that $L \in L_i$ and  $ y(t) \in Y_i(t)$ for all $ t \in \R $. 
		\item 
		Since $y \in Y_i^0$ it follows that its derivatives satisfy $ y^{(s)} \in Y_i^{s}$ for all $ 0 \leq s \leq S$.  
		\item 
		Let $\omega$ and $ c$ denote the frequency and Fourier coefficients of $y$ respectively. 
		If $X_i'$ is the output of Algorithm \ref{alg:FourierProjection} with input $M \in \N$, $L_i$ and $ \{ Y_i^{s} \}_{s=0}^{S}$, then 
		by Proposition \ref{prop:FourierProjection} it follows that $(\omega, \{c_k\}_{k=1}^\infty ) \in X_i'$.  
		\item 
		Let $ X_i''$ denote the output of Algorithm \ref{alg:TimeTranslate} with input $X_i'$. 
		By Theorem \ref{prop:TimeTranslation}, there exists a $ \tau \in \R$ such that the Fourier coefficients $ c'$ of $ y(t+\tau)$ satisfy  $(\omega, \{c'_k\}_{k=1}^\infty ) \in X_i''$.  
		\item 
		We have shown that if $y$ is a SOPS to \eqref{eq:Wright} 
		at parameter $ \alpha$ having frequency $ \omega$, then up to a time translation $( \alpha , \omega , c) \in \bigcup \cS'''$.  
		By Proposition 
		\ref{prop:Equivalence}  the SOPS to \eqref{eq:Wright}  at parameter $\alpha \in I_\alpha$ correspond to the non-trivial zeros of $F$ in $\bigcup \cS'''$.
		Hence, if there is a solution $ F(\hat{x}) =0$ for some $ x \in \R^2 \times \tilde{\oo}^s$ with $ \pi_\alpha (\hat{x})\in I_\alpha$, then $ \hat{x} \in \bigcup \cS'''$. 
		\item 
		Let $ \{flag,X_i^{(4)}\}$ denote the output of Algorithm \ref{alg:Prune} with input $X_i''' \in \cS'''$.
		By Theorem \ref{prop:Prune} we can replace each $X''' \in \cS'''$ with $X_i^{(4)}$, and it will still be the case that $\bigcup \cS'''$ contains all of the solutions to $ F=0$. 
		In particular, if $ flag=1$ then $X_i^{(4)} = \emptyset$ and we may remove $X_i'''$ in this case.  
		\item  If $(\alpha,\omega,c) \in \bigcup \cS'''$ and $ a_1 = [c]_1$, then by construction $ \omega \in I^\omega$ and $ a_1 \in I^{a_1}$.  
		As $I^\omega \times  I^{a_1} =  \bigcup_{(\beta_1,\beta_2) \in B} I^\omega_{\beta_1} \times  I^{a_1}_{\beta_2} $, then there is some $ (\beta_1, \beta_2) \in B$ such that $ (\omega,a_1) \in I^\omega_{\beta_1} \times I^{a_1}_{\beta_2}$. 
		\item
		As $ \bigcup_{X \in \cS'''} X \subseteq \bigcup_{\beta \in B} \tilde{X}_\beta$, then it follows that  $\bigcup_{X \in \cS'''} X \subseteq \bigcup_{\beta \in B} X_\beta $. That is to say  $\bigcup  \cS''' \subseteq \bigcup \cS$. 
		\item  
		Hence, $\bigcup \cS$ contains the Fourier coefficients of any possible SOPS. 
	\end{enumerate}
\end{proof}




\section{Global Algorithm}

\label{sec:GlobalAlgorithm}

After Algorithm \ref{alg:Comprehensive} has constructed a collection of cubes $\cS$ covering the solution space to $F=0$, we run a branch and  prune  algorithm. 
This algorithm iteratively inspects the elements in $X \in \cS$ and then constructs three new lists of cubes: $\cA$, $\cB$ and $\cR$.   
To summarize, first we compute the output $Prune(X) = \{flag,X'\}$  from Algorithm \ref{alg:Prune}. 
If $flag =1$, then there are no solutions in $X$, and we can remove $X$ from $\cS$.  
If $flag =2$, then the cube is in the neighborhood of the Hopf bifurcation, and we add $X'$ to $ \cB$. 
If $flag =3$, then for all $\alpha \in \pi_\alpha (X)$ there exists a unique solution to $F_\alpha = 0 $ in $X'$, and  we add $X'$ to $ \cA$. 
If $X'$ is too small, then we add it to $ \cR$. 
If the Krawczyk operator appears to be  effective at reducing the size of the cube, then the pruning operation is performed again. 
Otherwise $X'$ is subdivided along  some lower dimension and the resulting pieces are added back to $\cS$.

The most obvious difference between our algorithm and the classical algorithm is that we are working in infinite dimensions. 
While we store $2M+1$ real valued coordinates in a given cube, as in \cite{galias2007infinite,day2013rigorous} the subdivision is  only performed along a subset of these dimensions.  
Choosing which dimension to subdivide along can greatly affect the efficiency of a branch and bound algorithm, and there are heuristic methods for optimizing this choice \cite{csendes1997subdivision}. 
However since we are finding all the zeros along a 1-parameter family of solutions, these branching methods are not entirely applicable.
To determine which dimension to subdivide we select the dimension with the largest weighted diameter. That is, for a collection of weights $ \{ \lambda_i \}_{i=0}^{d}$ we define:  
\begin{equation}
	w(X,i) := 
	\begin{cases}
	\lambda_i \cdot \mbox{diam}\left( \pi_\alpha (X )\right) & \mbox{ if } i=0, \\
	\lambda_i \cdot \mbox{diam}\left(\left[\tilde{\pi}_M'(X )\right]_i\right) & \mbox{ otherwise.} \\
	\end{cases}
\end{equation}

\begin{algorithm}[Branch \& Bound]
	\label{alg:BranchAndPrune}
	Take as input a collection of cubes $\cS = \{ X_i \subseteq \R ^2 \times \tilde{\oo}^s\}$ with  $ M\geq 5$ and $s >2$, and as computational parameters: a halting criteria $\epsilon>0 $, a continue-pruning criteria $ \delta \geq0$,  	a maximum subdivision dimension $0 \leq d \leq 2M$ and a set of weights $\{\lambda_i\}_{i=0}^{d}$.
		The output is three lists of cubes: $ \cA, \cB$ and $\cR$. 
	\begin{enumerate}
		\item If $\cS$ is empty, terminate the algorithm.
		\item Select an element $X \in \cS$ and remove $X$ from $\cS$.  
		\item  
		Define $\{flag,X'\} = Prune(X)$ to be the output of Algorithm \ref{alg:Prune} with input $X$.
		\item If $flag=1$, then reject $X$ and GOTO Step 1.
		\item If $flag=2$, then add $X'$ to $\cB$ and GOTO Step 1.
		\item If $flag=3$, then add $ X'$ to $\cA$ and GOTO Step 1. 
		\item If $\max_{0\leq i\leq d} w(X',i) < \epsilon$, then add $X'$ to $\cR$ and GOTO Step 1. 
		\item Define $m=\lfloor d/2 \rfloor$. 
			If $(1+\delta ) < \frac{vol( \tilde{\pi}_m'(X))}{vol( \tilde{\pi}_m'(X') )}$,  then define $ X:= X'$ and GOTO Step 3.
		\item Subdivide $X'$ into two pieces, $X_1'$ and $X_2'$, along a dimension which maximizes $w(X',i)$, and so that $ X' = X_1' \cup X_2'$.   Add the two new cubes to $\cS$ and GOTO Step 1.
	\end{enumerate}
\end{algorithm}

\begin{theorem}  
	\label{prop:BnB}
		Let $ \cS = \{ X_i \subseteq \R ^2 \times \tilde{\oo}^s\}$ with $ M \geq 5$ and $ s>2$. 
	Let $\cA, \cB$ and $\cR$ be the output of  Algorithm \ref{alg:BranchAndPrune} run with input $\cS$ and various computational parameters. 
	\begin{enumerate}[(i)]
	\item   If $F(\hat{x}) =0$  for some $ \hat{x} \in \bigcup \cS$, then $ \hat{x}  \in  \bigcup \cA \cup  \cB \cup  \cR$. 
	\item For each $ X \in \cA$ and $ \alpha  \in \pi_\alpha(X)$,  there is a unique $ \hat{x}  = ( \alpha , \hat{\omega}_\alpha, \hat{c}_\alpha)\in X$ such that $ F(\hat{x}) =0$. 
	\item For each $ X \in \cB$, if there is a solution $ \hat{x} \in  X$ to $F=0$, then $\hat{x}$ is on the principal branch. 
	\end{enumerate}
\end{theorem}

\begin{proof}
	We prove the claims of the theorem. 
	\begin{enumerate}[(i)]
		\item  
		Suppose there is some solution $ \hat{x} \in X$ for some $X \in \cS$. 
		We show that $\hat{x} \in \bigcup \cS \cup \cA \cup \cB \cup \cR$ at every step of the algorithm. 		
		If we replace $X$ by $X'$ as in Step 3, then $\hat{x} \in X'$ by Theorem \ref{prop:Prune}. 
		In Step 4, if $ flag =1$ then in fact $ X' = \emptyset$, so $X$ could not have contained any solutions in the first place.  
		In Steps 5, 6 and 7, the cube $ X' $ is added to one of $ \cA$, $\cB$ or $\cR$. 
		Hence, as $\hat{x} \in X'$ then $ \hat{x} \in \bigcup \cS\cup  \cA \cup \cB \cup \cR$. 
		If in Step 8 we decide to prune the cube $X'$ again, then we may repeat the argument made for Steps 3-7. 
		In Step 9 we divide $X'$ into two new cubes $X_1'$ and $X_2'$ for which $ X' = X_1' \cup X_2'$. 
		Hence $\hat{x} $ will be contained in at least one of $ X_1'$ or $X_2'$, and both cubes are added to $ \cS$, so we cannot lose the solution in Step 9.

		Thus we have shown that  $\hat{x} \in \bigcup \cS \cup \cA \cup \cB \cup \cR$ at every step.		
		Since the algorithm can only stop when $\cS = \emptyset$, it follows that every solution $\hat{x}$ initially contained in $ \bigcup  \cS$ will eventually be contained in $ \bigcup \cA \cup \cB \cup \cR$. 
		
		\item  The only way a cube $X'$ can be added to $\cA$ is in Step 6. 
		That is, for some cube $ X \in \cS$ the output of Algorithm \ref{alg:Prune} returned $ \{3,X'\}$. 
		Thus, it  follows from Theorem \ref{prop:Prune} that for all $ \alpha  \in \pi_\alpha(X)$  there is a unique $ \hat{x}  = ( \alpha , \hat{\omega}_\alpha, \hat{c}_\alpha)\in X$ such that $ F(\hat{x}) =0$. 
		\item   The only way a cube $X'$ can be added to $\cB$ is in Step 5. 
		That is, for some cube $ X \in \cS$ the output of Algorithm \ref{alg:Prune} returned $ \{2,X'\}$. 
		Thus, it  follows from Lemma \ref{prop:BifNbd}  that the only solutions to $F=0$ in $X'$ are those on the principal branch. 
	\end{enumerate}
\end{proof}

If a cube has no zeros inside of it yet there is a solution close to its boundary, 
then proving that the cube does not contain any solutions can be very difficult, resulting in an excessive number of subdivisions. 
This phenomenon is common to branch and bound algorithms  and is referred to as the cluster effect \cite{schichl2004exclusion}.  
As we wish to enumerate not just isolated solutions but a 1-parameter family of solutions, the difficulty of the cluster effect is multiplied. 
Furthermore, we cannot expect that the boundary of a cube will almost never contain a solution. 
In particular, when we subdivide a cube we may also bisect the curve of solutions, and further subdivisions will not remedy this problem (see Figure \ref{fig:BranchANDBound}). 
As such,  we should not expect that $ \cR \neq \emptyset$.

To address this issue we apply Algorithm \ref{alg:Recombine} to the output of Algorithm \ref{alg:BranchAndPrune}.
In Step 1 we recombine cubes in $\cR$ which overlap in the $\alpha$ dimension. 
In Step 2 we split the cubes in $\cR$ along the $\alpha$-dimension to make them easier to prune, which we do in Step 3. 
Ideally by Step 4 all of the cubes have been removed from $ \cR$, having been added to either $ \cA$ or $\cB$.  

Even if $\cR = \emptyset$ at this point, it is not immediately clear  that the only solutions are on the principal branch. 
For two distinct cubes $ X_1,X_2 \in \cA$, if there is some $ \alpha_0$ such that $ \alpha_0 \in \pi_\alpha(X_1)$ and  $ \alpha_0 \in \pi_\alpha(X_2)$, then there could very well be two distinct solutions at the parameter $ \alpha_0$.    
In fact, since we subdivide along the $ \alpha$--dimension it is to be expected that a cube will share an $\alpha$--value with one or two other cubes. 
In Steps 6-9 of Algorithm \ref{alg:Recombine}  we   check to make sure that when two cubes have $ \alpha$--values in common, then there is a unique solution associated to each $\alpha_0 \in \pi_\alpha(X_1) \cap \pi_\alpha(X_2)$.

\begin{figure}[h]
	\centering
	\includegraphics[width = .475\textwidth]{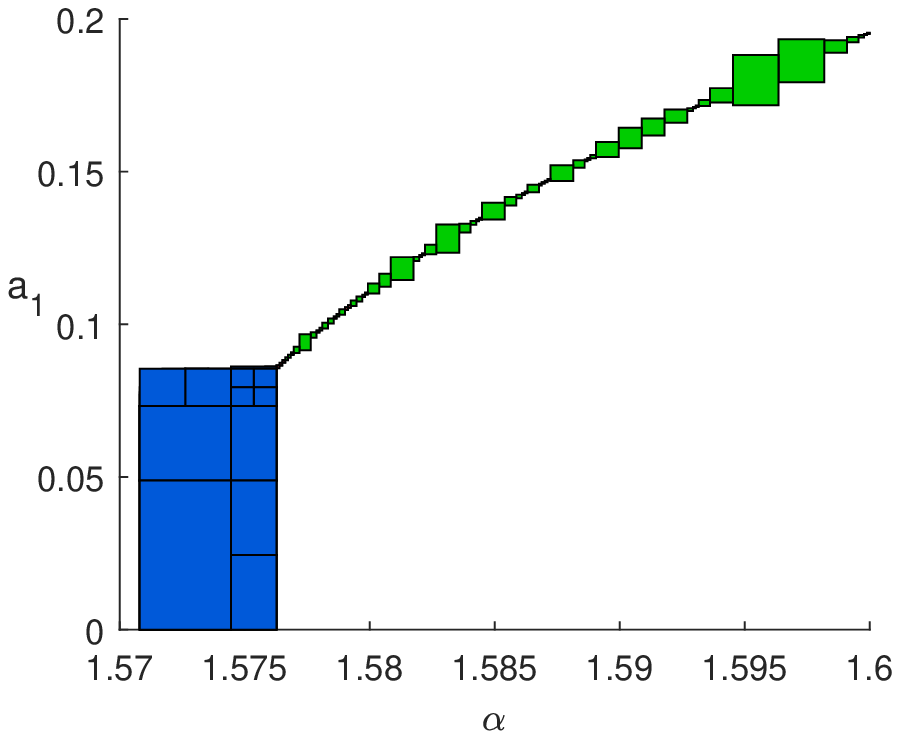}
	\includegraphics[width = .475\textwidth]{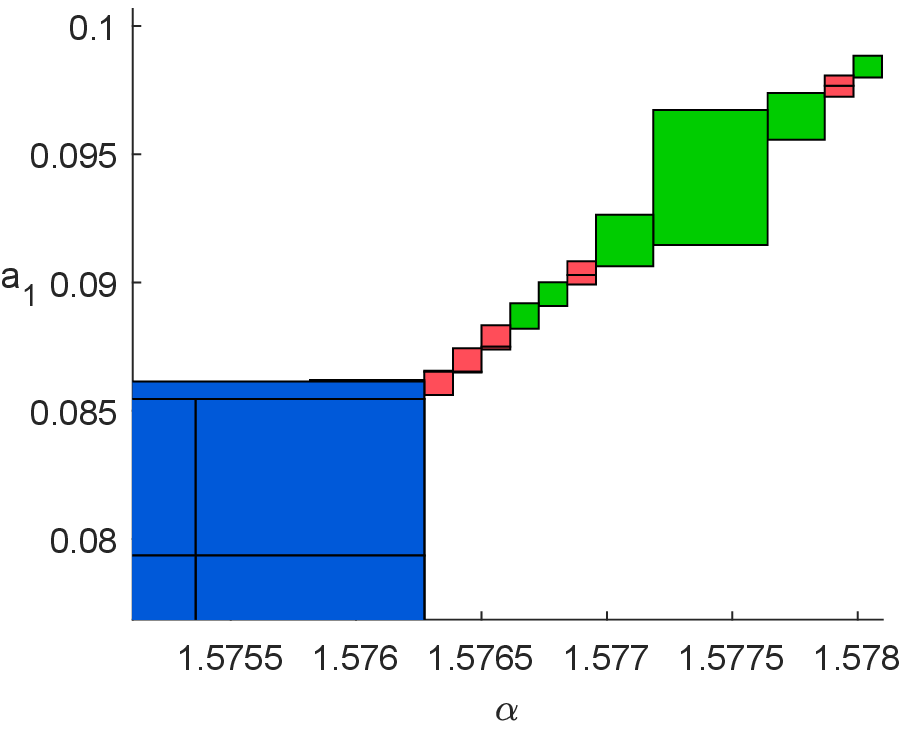}
	\caption{ \footnotesize
		An example output of Algorithm \ref{alg:BranchAndPrune}. The cubes in $\cA$ are in green, the cubes in $\cB$ are in blue, and the cubes in $ \cR$ are in red. }
	\label{fig:BranchANDBound}
\end{figure}

\begin{algorithm} 
	\label{alg:Recombine}
	Take as input   sets $ \cA, \cB, \cR$ produced by Algorithm \ref{alg:BranchAndPrune}  and a  computational parameter $n \in \N$.  
	The output is a pair of  intervals $ I_\alpha^\cA$, $I_\alpha^\cB$ and either success or failure. 
	\begin{enumerate}
		\item Combine the elements in $\cR$ whose $\alpha$-components overlap in more than just a point. 
		That is, for all $ X,Y \in \cR$, if $diam( \pi_\alpha(X) \cap \pi_\alpha(Y)) >0$, then replace $ X$ and $Y$ in the set $ \cR$ with a new cube $Z $ containing $ X \cup Y$. 
		\item Subdivide each $ X \in \cR$ along the $ \alpha$-dimension. 
		\item  
		For all $ X \in \cR$ calculate $ \{flag, X'\} = Prune^{(n)}(X)$, the output of Algorithm \ref{alg:Prune} iterated at most $n$ times with initial input $X$. 
		If $flag =1$, then remove $ X$ from $ \cR$. 
		If $flag =2$, then remove $ X$ from $ \cR$ and add $ X' $ to $\cB$. 
		If $flag =3$, then remove $ X$ from $ \cR$ and add $ X' $ to $\cA$.  
		\item 
		If $ \cR \neq \emptyset$ then return FAILURE. 
		\item 
		Define $I_\alpha^\cA =  \bigcup_{X \in \cA } \pi_\alpha(X)$ and $I_\alpha^\cB =  \bigcup_{X \in  \cB} \pi_\alpha(X)$.  
		\item 
		Construct a cover $ \cI_{\cB}'$ of the parts of cubes in $\cA$ which intersect with $ \bigcup \cB$. 
		That is, define $ \cI_{\cB} = \{ X \in \cA : \pi_\alpha (X) \cap I_\alpha^\cB \}$.  Then define $ \cI_{\cB}'$ by, for each $ X \in \cI_{\cB}$,  taking the $\alpha$-component of $X$ and setting it equal to $ \pi_\alpha(X) \cap I_{\alpha}^\cB$ and adding the modified cube to $\cI_{\cB}'$. 
		\item 
		For all $ X \in \cI_\cB'$ calculate $\{flag,X'\} = Prune^{(n)}(X)$, the output of Algorithm \ref{alg:Prune} iterated $n$--times with initial input $X$. 
		If $flag \neq 2 $ then return FAILURE.   
		\item 
		Construct a cover $ \cI_{\cA}'$ of the parts of cubes in $\cA$ which intersect with another cube in $\cA$.  
		That is, define $ \cI_{\cA} = \{ (X,Y) \in \cA \times \cA: X \neq Y,  \pi_\alpha (X) \cap  \pi_\alpha (Y) \neq \emptyset  \}$.  Then define $ \cI_{\cA}'$ by, for each $ (X,Y) \in \cI_{\cA}$,  defining a new cube $ Z$ which contains $ X \cup Y$, replacing the $\alpha$-component of $Z$ by $ \pi_\alpha(X) \cap \pi_\alpha(Y)$, and adding $Z$ to $\cI_{\cA}'$.

		\item 
		For all $ Z \in \cI_\cA'$ calculate $\{flag,Z'\} = Prune^{(n)}(Z)$, the output of Algorithm \ref{alg:Prune} iterated $n$--times with initial input $Z$. 
		If $flag \neq 3 $ then return FAILURE. 
		
		\item 
		If the algorithm reaches this point, return SUCCESS.				
	\end{enumerate}
\end{algorithm}

\begin{theorem}
	Let $ \cA,\cB,\cR$ denote the output of Algorithm \ref{alg:BranchAndPrune} run with input $ \cS= \{ X_i \subseteq \R ^2 \times \tilde{\oo}^s  \}$ where  $M \geq 5$ and $ s>2$. 
	Suppose having received input $ \cA,\cB,\cR$  and $ n \in \N$, Algorithm \ref{alg:Recombine} returns SUCCESS and intervals  $I_{\alpha}^\cA$ and $I_{\alpha}^\cB$. 
\begin{enumerate}[(i)]
\item If $ \alpha \in I_\alpha^\cA \backslash I_\alpha^\cB$, then there is a unique solution $\hat{x}_\alpha = (\alpha , \hat{\omega}_\alpha,\hat{c}_\alpha) \in \bigcup \cS$ such that $ F_\alpha (\hat{\omega}_\alpha,\hat{c}_\alpha)=0$.  
\item  If $ \alpha \in I_\alpha^\cB$, then the only solutions to $F_\alpha = 0$ in $\bigcup \cS$ are on the principal branch. 

\end{enumerate}

	\label{prop:Rescale}
\end{theorem}
\begin{proof}
	We describe the first 4 steps of the algorithm and then prove the theorem.  

	\begin{enumerate}
		\item Let $\cR$ denote the initial input to the algorithm and $ \cR'$ denote the resulting set produced by Step 1. 
		By its construction, it follows that  $ \bigcup \cR \subseteq  \bigcup \cR'$. 
		\item If we subdivide the cubes in $\cR'$, then it is still true that $ \bigcup \cR \subseteq \bigcup \cR'$. 
		\item  
		As described in the proof of Theorem \ref{prop:BnB}, if $flag = 1,2,3$ then it is appropriate to respectively, discard $X$, add $X'$ to $\cB$ and add $X'$ to $\cA$. 
		Appropriate, that is, in the sense that the conclusion of Theorem \ref{alg:BranchAndPrune} will hold for these modified sets $ \cA$, $\cB$ and $ \cR$. 
		\item  If we cannot show that every region of phase-space lies in either $\cA$ or $\cB$ then we are unable to prove the theorem. 
		Otherwise, every solution to $ F=0$ in $\bigcup \cS$ is contained in $\bigcup \cA \cup \cB$. 
			\end{enumerate}

		We prove claim $(i)$. 
		If $ \alpha \in I_\alpha^\cA \backslash I_\alpha^\cB$   there is a solution $\hat{x}_\alpha$ to $ F_\alpha =0$ in $ \bigcup \cA $. 
		Suppose there exists a second distinct solution $ \hat{x}_\alpha'$ to $F_\alpha =0$. 
		Since each cube  $X \in \cA$ contains a unique solution for all $\alpha \in \pi_\alpha (X)$, there would exist distinct cubes $X,Y \in \cA$ such that $ \hat{x}_\alpha \in X$ and  $ \hat{x}_\alpha' \in Y$.  
		It follows then that there exists some $Z \in \cI_\alpha'$ such that $ \hat{x}_\alpha,\hat{x}_\alpha'\in Z$.  
		Since it is		
		determined by Step 9 that $ flag=3$ in the output of $Prune^{(n)}(Z)$, therefore by Theorem \ref{prop:Prune} there exists a unique solution  to $F=0$ in $Z$.  
		Thereby $\hat{x}_\alpha = \hat{x}_\alpha'$, and if $ \alpha \in I_\alpha^\cA \backslash I_\alpha^\cB$, then there is a unique solution $\hat{x}_\alpha = (\alpha , \hat{\omega}_\alpha,\hat{c}_\alpha) \in \bigcup \cS$ such that $ F_\alpha (\hat{\omega}_\alpha,\hat{c}_\alpha)=0$.

		We prove claim $(ii)$. 
		Suppose there exists some $\hat{x}_\alpha$ such that  $ \alpha \in I_\alpha^\cB$ and  $ F_\alpha (\hat{\omega},\hat{c})=0$. 
		Since the algorithm passed through Step 4, it follows that  $\hat{x}_\alpha \in \bigcup \cA \cup \cB$. 
		If $ \hat{x}_\alpha \in \bigcup \cB$, then $\hat{x}_\alpha$  is on the principal branch by Theorem  \ref{prop:BnB}.
		If $ \hat{x}_\alpha \in \bigcup \cA$, then there exists a cube $X \in \cI_\cB'$ such that $ \hat{x}_\alpha \in X$.  
		If the Algorithm \ref{alg:Recombine} is successful, then when Algorithm  \ref{alg:Prune}  is run $n$--times with initial input $X$ it will produce $flag =2$. 
		Hence by Theorem \ref{prop:Prune} this solution $\hat{x}_\alpha \in \bigcup \cA$ must be on the principal branch.  

\end{proof}

\begin{proof}[Proof of Theorem \ref{prop:Jones}] 
	We implemented the algorithms discussed in this paper using MATLAB version R2017b (see \cite{JonesCode} for the code). 
	The calculations were performed on Intel Xeon E5-2670 and Intel Xeon E5-2680 processors, and used INTLAB for the interval arithmetic \cite{rump1999intlab}. 
	A summary of the algorithms' runtime is  given in Table  \ref{table:RunTimes}.

	For the intervals $I_\alpha$ taking the values (containing at least)  $ [\pp,1.6]$, $ [1.6,1.7]$, $ [1.7,1.8]$, and $ [1.8,1.9]$,
	we ran \cite[ Algorithm 5.1]{jlm2016Floquet} using computational parameters $i_0=2$, $j_0=20$, $n_{Time} = 32$, $N_{Period} =10$, $N_{Prune} =4$, $\epsilon_1 = 0.05$ and $\epsilon_2 = 0.05$. 
	We then ran Algorithm \ref{alg:Comprehensive} using computational parameters   $M=10$ and $S=3$, and $N=15$ producing outputs $\cS_{I_\alpha}$ (see Figure \ref{fig:InitialPrep}). 
	By Theorem \ref{prop:Comprehensive}, if $y$ is a SOPS at parameter $ \alpha \in I_\alpha$ given as in \eqref{eq:FourierEquation}, then $ (\alpha,\omega,c) \in \bigcup \cS_{I_\alpha}$. 
	By Proposition \ref{prop:Equivalence}  	the SOPS to \eqref{eq:Wright} at parameters $ \alpha \in I_\alpha$ are in bijective correspondence with the nontrivial zeros of $F$ inside $ \bigcup \cS_{I_\alpha}$. 
	
	On each of the collections of cubes $\cS_{I_\alpha}$ we ran Algorithm \ref{alg:BranchAndPrune}, using the following computational parameters: 
	For the stopping criterion we used $ \epsilon = 0.0001$ for  $\alpha \in  [\pp,1.6]$ and   $ \epsilon = 0.01$ otherwise. 
	For the continue-pruning criterion, in every case  we used $\delta = 0.5$.
	For the maximal subdivision dimension, in each case we used $d =6$, corresponding to the  variables  $ \alpha,\omega,a_1 \in \R $ and $ c_2 , c_3 \in \C$.
	For the set of weights, in each case we used  $ \lambda_{0} = 8$ (corresponding to $\alpha$) and $ \lambda_i = 1$ otherwise. 
	
	The output of Algorithm \ref{alg:BranchAndPrune} are sets $\cA_{I_\alpha},\cB_{I_\alpha},\cR_{I_\alpha}$. 
	On each of these resulting outputs we ran Algorithm \ref{alg:Recombine} using $n=5$, and in each case it was successful, producing sets $ I_\alpha^\cA$ and $I_\alpha^\cB$. 
	When $I_\alpha =  [\pp,1.6]$ then $I_\alpha^\cB =[\pp,\pp + 0.00550]$ and $I_\alpha^\cA  = [\pp + 0.00550,1.6]$, and otherwise $I_\alpha^\cA  = I_\alpha$.  
	By Theorem \ref{prop:BnB}, this shows that for all $ \alpha \in [\pp + 0.00550,1.9]$ there exists a unique solution to $ F_\alpha =0$ in $ \bigcup \cS$, and if $ \alpha \in [\pp , \pp + 0.00550]$ then the only solutions that exist are on the principal branch. 
	Note that by \cite{BergJaquette} there are no solutions at $ \alpha = \pp$ on or off the principal branch, and there are no folds in the principal branch for $ \alpha \in ( \pp, \pp+0.00553]$.  
	 Hence for all $ \alpha \in ( \pp , 1.9]$ there exists a unique solution to \eqref{eq:Wright}.  	 	
	By \cite{jlm2016Floquet} and \cite{xie1991thesis} there exists a unique SOPS to \eqref{eq:Wright}  for $ \alpha \in [1.9,6.0]$ and $ \alpha \geq 5.67$ respectively. 
	Hence there exists a unique SOPS to \eqref{eq:Wright} for all $ \alpha > \pp$. 
\end{proof}

\begin{center}

\begin{tabular}{ r r r r r r r r }
	\hline
$I_\alpha$  & $N_{bf}$ 	&$N_{grid}'$	&$N_{grid}$	& $T_{bf}$	& $T_{grid}$ & $T_{bb}^*$ 	& $T_{verify}$ \\
\hline 
$[\pp,1.6]$ & 1604 		&614			&181		& 602.5		& 5.2  		&	$2.7^{*}$	& 1.4 \\
$[1.6,1.7]$ & 985 		&861			&165 		& 461.6		& 5.6		&	$4.6^*$		& 1.2 \\
$[1.7,1.8]$ & 604 		&566			&143 		& 335.1		& 3.6		&	$10.1^*$	& 0.4 \\
 $[1.8,1.9]$& 292 		&277			& 97 		& 135.6		& 1.9		& 	$67.0^*$	& 0.6 \\
\hline\end{tabular}
	\captionof{table}{\footnotesize Computational benchmarks from the computer-assisted proof of Theorem \ref{prop:Jones}. 
		$N_{bf}$ --	the number of bounding functions output  by \cite[ Algorithm 5.1]{jlm2016Floquet}.
		$N_{grid}'$ -- the number of cubes in $\mathcal{S}'''$ after Step 7 in Algorithm \ref{alg:Comprehensive}. 
		$N_{grid}$ -- the number of cubes output by Algorithm \ref{alg:Comprehensive}.
		$T_{bf}$ -- the run time (min.) of \cite[ Algorithm 5.1]{jlm2016Floquet}.
		$T_{grid}$ -- the run time (min.) of Algorithm \ref{alg:Comprehensive}.
		$T_{bb}^*$ 	-- the run time (min.) of Algorithm \ref{alg:BranchAndPrune} parallelized using 20 workers. 
		\textbf{$T_{verify}$} -- the run time (min.) of Algorithm \ref{alg:Recombine}.
	}
\label{table:RunTimes}
\end{center}

\begin{proof}[Proof of Theorem \ref{prop:Rescaling}]
	By \cite{mallet1988morse} every global solution to \eqref{eq:MNF} has a positive, integer valued lap number $V(x,t)$. 
	For non-zero $x$ the lap number will be an odd integer,  defined by fixing the smallest possible $ \sigma \geq t$ such that $ x(\sigma)=0$ and defining: 
	\begin{equation*}
	V(x,t) = 
	\begin{cases}
	\mbox{the \# of zeros (counting multiplicity) of $x(s)$ in $(\sigma-1,\sigma]$; or } \\
	\mbox{$1$ if no $\sigma $ exists}.
	\end{cases}
	\end{equation*}

	Let us fix $x_0$ as a periodic solution to \eqref{eq:MNF} with period $L_0$. 
	For any $t\in \R $ the lap number $ V(x_0,t)$ remains constant, and we can define $ N:= V(x_0,t)$. 
	If $ N=1$ then $x_0$ must be a SOPS. 	
	If $ N \geq 3$ then define the integer $n:= \tfrac{N-1}{2}$ and  $r := 1 - n L_0 $. 
	By \cite{mallet1988morse}, it follows that $2/N < L_0 < 2/(N-1)$, hence $ 0<r<N^{-1}$. 
	Defining $x_1(t):= x_0(r t)$ and  $ \alpha_1 = r \alpha_0 $ we calculate the derivative of $ x_1(t)$ as: $	x_1'(t)= -  \alpha_1 f( x_0(rt -1))$.
	We may further compute: 
	\begin{align*}
	x_0(rt-1) = x_0(rt -1 + n L_0) = x_0(r(t-1)) = x_1(t-1).
	\end{align*}
	Hence it follows that $ x_1'(t) = - \alpha_1 f( x_1(t-1))$. 
	Thus we have shown that if $V(x_0) \geq 3$  then $x_0$ is a rescaling of a periodic solution $x_1$ with period length $L_1 = L_0/r >2$.  
	Hence $x_0$ is a rescaling of a SOPS. 
	
\end{proof}

\section{Future Work}
\label{sec:FutureWork}

One pertinent question that remains concerns the period length of SOPS to Wright's equation. 
\begin{conjecture}
	\label{prop:LengthConj}
	The period length of SOPS to \eqref{eq:Wright} increases monotonically in $\alpha$.
\end{conjecture}
The rigorous numerics in \cite{lessard2010recent,jlm2016Floquet} strongly suggests this to be true when $ \alpha \leq 6$, and when $ \alpha \geq 3.8$  the period length $L$   satisfies $|L - \alpha^{-1} e^\alpha| < 7.66 \alpha^{-1}$ by \cite{nussbaum1982asymptotic}.
It is known that the period length increases monotonically when $ \alpha \in ( \pp, \pp+6.830 \times 10^{-3}]$ by   \cite{BergJaquette}. 
However Conjecture \ref{prop:LengthConj} is unresolved for $ \alpha > \pp + 6.830 \times 10^{-3}$. 

Another question, proposed in \cite{neumaier2014global},  is the generalized Wright's conjecture.

\begin{conjecture}
	\label{prop:GenWright}
	For every $ \alpha >0$ the set $\overline{U(\alpha)}$,  the closure of the forward extension by the semiflow of a local unstable manifold at zero,  is the global attractor for \eqref{eq:Wright}. 
\end{conjecture}

This is known to be true for $\alpha \leq \pp $ by \cite{wright1955non,neumaier2014global,BergJaquette} and is unresolved for $ \alpha > \pp$. 
Conjecture \ref{prop:GenWright} can be reduced to a question about the number of rapidly oscillating periodic solutions, and moreover Conjecture  \ref{prop:LengthConj} implies Conjecture \ref{prop:GenWright}.  
To wit, by the Poincar\'e-Bendixson theorem for monotone feedback systems \cite{mallet1996poincare}, the $\omega$-limit set of any initial data to \eqref{eq:Wright} is either $ 0$ or a periodic orbit.   
The lap number organizes the attractor into Morse sets $S_N$ by \cite{mallet1988morse}, and by \cite{fiedler1989connections} there is always a connecting orbit from the unstable manifold of the origin to the Morse set $ S_N$. 
Hence, to prove Conjecture \ref{prop:GenWright},  it would suffice to show that each Morse set consists of exactly one periodic orbit.  

By Theorem \ref{prop:Rescaling} there are no isolas of periodic orbits, so multiple rapidly  oscillating periodic solutions can only arise if there is a fold in one of the branches of rapidly oscillating periodic solutions. 
If Conjecture \ref{prop:LengthConj} holds, then such a fold can be ruled out using rescaling equation in Theorem  \ref{prop:Rescaling}. 
In particular, if there are two SOPS at parameters $ \alpha_1 < \alpha_2$ with period lengths $L_1, L_2$ and the equality  $ \alpha_0 = \alpha_1 ( 1 + n L_1) = \alpha_2 (1 + n L_2)$ holds, then there will be two distinct  rapidly oscillating periodic solutions at parameter  $\alpha_0$. 
This equality cannot hold if $L_1 < L_2$ whenever $ \alpha_1 <\alpha_2$. 
Thereby Conjecture \ref{prop:LengthConj} implies Conjecture \ref{prop:GenWright}.

There are still further questions about Wright's equation. 
In \cite{mccord1996global} the authors show a semi-conjugacy of Wright's equation, and negative feedback systems more generally, onto a family of finite dimensional ODEs. 
Outside the dynamics described by this semi-conjugacy, are there any other interesting dynamics in \eqref{eq:Wright}?
Furthermore, do the stable and unstable manifolds of the periodic orbits in \eqref{eq:Wright} intersect transversely? 
\newline

There are many future directions for the  rigorous numerics of infinite dimensional dynamical systems.  
Perhaps one of the most striking features of Figure \ref{fig:Verified} and Figure \ref{fig:BranchANDBound} is the non-uniform size of cubes. 
This seems to be a result of  applying the branch and bound method to a 1-parameter family of solutions instead of a collection of  isolated solutions.  
One approach would be to first validate a neighborhood around the branch of solutions (\emph{\'a la} \cite{lessard2010recent}) and then use a branch and bound method to ensure that there are no solutions outside of this neighborhood. 
In this paper, we used a collection of weights $\{ \lambda\}_{i=0}^d$ to mitigate this problem.  
When using all equal weights ($\lambda_i=1$ for all $i$), the vast majority of cubes output by Algorithm \ref{alg:BranchAndPrune} ended up in $\cR$.  
Having a better heuristic for deciding along which dimension to branch would be very useful,  
particularly so if it does away with the \emph{a priori} need to select a maximal subdivision dimension $d$ as a computational parameter.

Integral to the success of our algorithm (allowing it to finish in finite time) are the estimates derived in \cite{jlm2016Floquet} which bound \emph{all} of the slowly oscillating periodic solutions to Wright's equation.  
Since most initial conditions are attracted to the single SOPS in Wright's equation, it was sufficient for the methods in \cite{jlm2016Floquet} to be relatively simple.  
Future work could be done toward bounding all periodic orbits when there are multiple (unstable) solutions, or when the dimension is higher, as well as bounding all periodic solutions to ODEs and PDEs.

Another question,  explored in \cite{lessard2017computer}, is ``what the best Banach space to work in?''  
In this paper we consider the space $ \oo^s$ of Fourier coefficients with algebraic decay. 
In Algorithm \ref{alg:Comprehensive}, the estimates for obtaining \emph{a priori} estimates on the Fourier coefficients of SOPS always improve in absolute terms by using larger value of $S$. 
However, the value of $C_0$ will increase when using a larger $S$. 
It would likely be beneficial to initially run Algorithm \ref{alg:Comprehensive} with a large $S$, and then convert these bounds into a smaller $S$ so that $ C_0$ will shrink as well. 
However, for other applications and other infinite dimensional problems, the question of what is the optimal Banach space remains.

\section*{Acknowledgments}
The author thanks Konstantin Mischaikow for many insightful discussions, as well as John Mallet-Paret and Roger Nussbaum for discussions on future work.  
\newline

\noindent The author acknowledges the Office of Advanced Research Computing (OARC) at Rutgers, The State University of New Jersey for providing access to the Amarel cluster and associated research computing resources that have contributed to the results reported here.


\bibliographystyle{abbrv}
\bibliography{BibWright}

\end{document}